\documentclass[11pt]{article}
\usepackage[total={6.5in, 9in}]{geometry}
\usepackage{mathrsfs}
\usepackage{t1enc}
\usepackage[latin1]{inputenc}
\usepackage[english]{babel}
\usepackage{amsmath,amsthm,amssymb,thmtools}
\usepackage{amsfonts}
\usepackage{latexsym}
\usepackage{enumerate}
\usepackage{dsfont}

\usepackage{euscript}
\usepackage{graphicx}
\usepackage[natural]{xcolor}
\usepackage{tikz}
\usepackage{tikz-3dplot}
\usetikzlibrary{shapes}
\usetikzlibrary{arrows,decorations.pathmorphing,backgrounds,positioning,fit,petri,automata}

\usepackage[
pdfauthor={},
pdftitle={},
pdfstartview=XYZ,
bookmarks=true,
colorlinks=true,
linkcolor=blue,
urlcolor=blue,
citecolor=blue,
bookmarks=true,
linktocpage=true,
hyperindex=true
]{hyperref}

\theoremstyle{plain}
\newtheorem{thm}{Theorem}[section]

\newtheorem{prop}[thm]{Proposition}
\newtheorem{lemma}[thm]{Lemma}

\newtheorem{claim}[thm]{Claim}
\newtheorem{assum}{Assumption}

\numberwithin{equation}{section}

\newcommand{\overconnect}{
  \raisebox{-0.4ex}{$\overset{\scriptscriptstyle \hspace{0.4em} >}{\scalebox{0.8}{$\fontfamily{cmex}\selectfont\hookrightarrow$}}$}
  \hspace{0.4em}
}

\newcommand{\underconnect}{
  \raisebox{-0.4ex}{$\overset{\scriptscriptstyle \hspace{0.4em} <}{\scalebox{0.8}{$\fontfamily{cmex}\selectfont\hookrightarrow$}}$}
  \hspace{0.4em}
}

\begin{document}
\title{An Ore-type condition for $H$-tilings in graphs}
\author{Yuping Gao\footnote{Lanzhou University, School of Mathematics and Statistics, Lanzhou 730000, China.
		Email: {\tt gaoyp@lzu.edu.cn}.
		Supported in part by  NSFC grant, No. 12271228 and China Scholarship Council, No. 202406180027.}
	\qquad
	Yilin Guo\footnote{Shandong University, School of Mathematics, Jinan 250100, China.
		Email:	{\tt {doudou752991@163.com}}.
		}\qquad Guanghui Wang\footnote{a. Shandong University, School of Mathematics, Jinan 250100, China.
 b. Shandong University, State Key Laboratory of Cryptography and Digital Economy Security, Jinan 250100, China.
Email: {\tt ghwang@sdu.edu.cn}. Supported in part by  NSFC grant, No. 12231018 and National Key Research and Development Program, No. 2023YFA1009603.}\qquad
Lin-Peng Zhang\footnote{Shandong University, School of Mathematics, Jinan 250100, China. Email: {\tt lpzhangmath@163.com}. Supported by the China Postdoctoral Science Foundation (No. 2025M783120), and Shandong Postdoctoral Science Foundation (No. SDZZ-ZR-202501340). }
}

\date{}
\maketitle

\begin{abstract}
A graph $G$ admits an $H$-tiling if it contains a collection of vertex-disjoint copies of $H$. In this paper, we confirm a conjecture proposed by K\"{u}hn, Osthus, and Treglown by showing that for any given graph $H$, there exists a constant $C(H)$ such that the following holds. If $G$ is a sufficiently large $n$-vertex graph satisfying $d(x) + d(y) \geq 2\left(1 - 1/\chi_{\text{cr}}(H)\right)n$ for all nonadjacent vertices $x, y \in V(G)$, then $G$ contains an $H$-tiling covering all but at most $C(H)$ vertices. Here $\chi_{\text{cr}}(H)$ denotes the critical chromatic number of $H$.
\medskip

\medskip

\end{abstract}

\section{Introduction}\label{sec:Introduction}

\subsection{Tilings in graphs of large minimum degree}\label{subsec:md}

Given two graphs $H$ and $G$, an $H$-\emph{tiling} in $G$ is a collection of vertex-disjoint copies of $H$ in $G$. $H$-tilings are generalizations of matchings, which correspond to the case where $H$ is a single edge. An $H$-tiling is called \emph{perfect} (or an $H$-\emph{factor}) if it covers all vertices of $G$.

While Tutte's theorem provides a complete characterization for the existence of a $K_2$-factor, no analogous characterization exists for most other connected graphs $H$. Indeed, Hell and Kirkpatrick~\cite{hell1983} showed that the problem of deciding whether a graph $G$ has an $H$-factor is NP-complete precisely when $H$ contains a component with at least three vertices. Consequently, research has focused on identifying sufficient conditions that guarantee the existence of an $H$-factor. A foundational result in this direction is the Hajnal-Szemer\'{e}di theorem.

\begin{thm}[Hajnal and Szemer\'{e}di~\cite{hajnal1970}]\label{thm:HS1970Kr-tiling}
If $G$ is an $n$-vertex graph, $r\mid n$, and $\delta(G) \ge (1 - 1/r)n$, then $G$ has a $K_r$-factor.
\end{thm}

In the 1990s, Alon and Yuster generalized Theorem~\ref{thm:HS1970Kr-tiling} to arbitrary graphs $H$ by incorporating the chromatic number $\chi(H)$.

\begin{thm}[Alon and Yuster~\cite{alon1992}]\label{thm:AY1992}
For every $\varepsilon>0$ and for every integer $h$, there
exists $n_0=n_0 (\varepsilon, h)$ such that for every graph $H$ on $h$ vertices, any graph $G$ on $n>n_0$ vertices and $\delta(G) \ge(1 - 1/\chi(H))n$
contains at least $(1 - \varepsilon)n/h$ vertex-disjoint copies of $H$.
\end{thm}

\begin{thm}[Alon and Yuster~\cite{alon1996}]\label{thm:AY1996H-factor}
For every $\varepsilon>0$ and for every integer $h$, there
exists $n_0=n_0 (\varepsilon, h)$ such that for every graph $H$ on $h$ vertices, any graph $G$ on $n>n_0$ vertices and $\delta(G) \ge (1 -1/\chi(H) + \varepsilon)n$, has an $H$-factor.
\end{thm}

Alon and Yuster conjectured that the error terms involving $\varepsilon n$ in Theorems~\ref{thm:AY1992} and~\ref{thm:AY1996H-factor} could be replaced by constants. In~\cite{alon1992}, they remarked that this is essentially best possible. These conjectures were resolved by Koml\'{o}s, S\'{a}rk\"{o}zy, and Szemer\'{e}di.

\begin{thm}[Koml\'{o}s, S\'{a}rk\"{o}zy, and Szemer\'{e}di~\cite{komlos2001}]\label{thm:KKS2001H-factor}
For every graph $H$, there exists a constant $C$ such that if $G$ is an $n$-vertex graph with
$\delta(G) \ge (1 -1/\chi(H))n + C$,
then $G$ has an $H$-factor.
\end{thm}

\begin{thm}[Koml\'{o}s, S\'{a}rk\"{o}zy, and Szemer\'{e}di~\cite{komlos2001}]\label{thm:KSS2001H-matching}
Given the conditions of Theorem~\ref{thm:KKS2001H-factor}, if
$\delta(G) \ge (1 -1/\chi(H))n$,
then $G$ has an $H$-tiling that covers all but at most $C$ vertices.
\end{thm}

K{\"u}hn and Osthus~\cite{kuhn2006, kuhn2009} later showed that for any graph $H$, either the standard chromatic number or the critical chromatic number is the relevant parameter. The \emph{critical chromatic number} $\chi_{cr}(H)$ of a graph $H$ is defined as
\[
\chi_{cr}(H)= (\chi(H) - 1)\frac{|H|}{|H| - \sigma(H)},
\]
where $\sigma(H)$ denotes the size of the smallest color-class in any $\chi(H)$-coloring of $H$. Throughout this paper, we consider only graphs $H$ containing at least one edge (without explicit mention), ensuring $\chi_{cr}(H)$ is well-defined. Note that $\chi(H)-1 < \chi_{cr}(H) \leq \chi(H)$ for all graphs $H$, with equality $\chi_{cr}(H)=\chi(H)$ holding precisely when every $\chi(H)$-coloring of $H$ has color-classes of equal size.

In the context of almost perfect tilings, Koml\'{o}s was the first to demonstrate the relevance of the critical chromatic number.

\begin{thm}[Koml\'{o}s~\cite{komlos2000}]\label{thm:K2000H-matching}
For any graph $H$ and $\varepsilon > 0$, there exists a threshold $n_0 = n_0(H, \varepsilon)$ such that if $n \ge n_0$ and an $n$-vertex graph $G$ satisfies $\delta(G) \ge (1 - 1/\chi_{cr}(H))n$,
then $G$ contains an $H$-tiling that covers all but at most $\varepsilon n$ vertices.
\end{thm}

Koml\'{o}s also conjectured that the number of uncovered vertices can be reduced to a constant depending only on $H$. This was proved by Shokoufandeh and Zhao.

\begin{thm}[Shokoufandeh and Zhao~\cite{zhao2003}]\label{thm:SZ2003H-matching}
For any $k$-chromatic graph $H$ on $h$ vertices with smallest color-class of
order $\sigma$, there exists an $n_0$ such that, for all $n\geq n_0$, if $G$ is an $n$-vertex graph with
$\delta(G) \ge (1 - 1/\chi_{cr}(H))n$,
then $G$ contains an $H$-tiling that covers all but at most $\dfrac{5(k-2)(h-\sigma)^2}{\sigma(k-1)}$ vertices of $G$.
\end{thm}

\subsection{Ore-type conditions for tilings in graphs}\label{subsec:ored}

Beyond the minimum-degree conditions discussed above, it is natural to seek other degree-based criteria that force the existence of $H$-tilings. One of the most important such criteria is the \emph{Ore-type} condition, which gives a lower bound on the sum of the degrees of any two nonadjacent vertices. This notion originates from Ore's theorem~\cite{ore1960}, which asserts that a graph $G$ on $n \ge 3$ vertices contains a Hamiltonian cycle whenever
$d(x) + d(y) \ge n$
for every pair of nonadjacent vertices $x,y\in V(G)$.
A notable consequence of Kierstead and Kostochka's work on equitable colorings~\cite{kierstead2008} yields an Ore-type refinement of Theorem~\ref{thm:HS1970Kr-tiling}.

\begin{thm}[Kierstead and Kostochka~\cite{kierstead2008}]\label{thm:K2008Kr-factor}
	If $G$ is an $n$-vertex graph with $ r \mid n $ and $d(x) + d(y) \ge 2(1 - 1/r)n - 1$ for all nonadjacent $x, y \in V(G)$, then $G$ contains a $K_r$-factor.
\end{thm}

Cheng, Li, Sun, and Wang~\cite{2511.03957} proved that if $d(x) + d(y) \ge 2(1 - 1/r)n - 2$ for all nonadjacent $ x, y \in V(G) $, then the graph $ G $ either contains a $ K_r $-factor or belongs to one of two extremal graph classes.
In a related direction, Kawarabayashi~\cite{kawarabayashi2002} investigated the Ore-type condition that guarantees a $K_4^-$-tiling covering a specified number of vertices, where $K_4^-$ denotes the graph obtained  from $K_4$ by deleting one edge.
Extending these results to general graphs $H$, K{\"u}hn, Osthus, and Treglown~\cite{kuhn2009ore} established an Ore-type analogue of Theorem~\ref{thm:K2000H-matching} regarding almost perfect tilings.

\begin{thm}[K\"{u}hn, Osthus, and Treglown~\cite{kuhn2009ore}]\label{thm:AYTreglown}
For every graph $H$ and each $\eta> 0$, there exists an integer $n_0=n_0(H,\eta)$ such that if $G$ is a graph on $n\geq n_0$ vertices and \[d(x) + d(y) \geq 2\left(1 - \frac{1}{\chi_{\mathrm{cr}}(H)}\right)n\]
for all nonadjacent $x\neq y\in V (G)$, then $G$ has an $H$-tiling covering all but at
most $\eta n$ vertices.
\end{thm}

Motivated by Theorem~\ref{thm:SZ2003H-matching}, which states that the number of uncovered vertices in the minimum degree setting can be reduced to a constant depending only on $H$, K{\"u}hn, Osthus, and Treglown~\cite{kuhn2009ore} conjectured that the same should hold for the Ore-type condition. By adapting the proof techniques developed in~\cite{zhao2003}, we confirm this conjecture.

\begin{thm}\label{thm:main}
For every graph $H$, there exists a constant $C= C(H)$ and an integer $n_0$ such that if $G$ is a graph on $n\geq n_0$ vertices and
\[
d(x) + d(y) \geq 2\left(1 - \frac{1}{\chi_{\mathrm{cr}}(H)}\right)n
\]
for all nonadjacent $x, y \in V(G)$, then $G$ has an $H$-tiling covering all but at
most $C$ vertices.
\end{thm}

\subsection{Notation}\label{subsec:notation}

Let $G$ be a graph. For a vertex subset $X\subseteq V(G)$, let $G[X]$ be the subgraph of $G$ induced on $X$ and $G-X=G[V(G)\setminus X]$. For a vertex $v\in V(G)$, we write $G-v$ for $G-\{v\}$. Denote the set of neighbors of $v$ in $X$ by $N_G(v, X)$, and let $d_G(v, X) = |N_G(v, X)|$ be the degree of $v$ in $X$. When $X = V(G)$, we use the standard abbreviations $N_G(v)$ and $d_G(v)$, respectively.
For two subsets $A, B \subseteq V(G)$, we define $e_G(A, B)=\sum_{v\in A}d_G(v,B)$. Note that when $A$ and $B$ are disjoint, $e_{G}(A,B)$ is the number of edges in $G$ with one endvertex in $A$ and the other in $B$.
For a directed graph $D$, let $N_{D}^{+}(v) = \{u \in V(D) : (v, u) \in E(D)\}$ denote the out-neighborhood of $v$, and let $d_{D}^{+}(v) = |N_{D}^{+}(v)|$ denote the out-degree of $v$. We omit the subscripts $G$ and $D$ when the graph or digraph is clear from the context.

A bipartite graph with vertex partition $(A, B)$ is denoted by $G(A, B)$. We use $K_{n_1, n_2, \ldots, n_r}$ to denote the complete $r$-partite graph with partition sizes $n_1, n_2, \ldots, n_r$. In particular, $K_r(s)$ denotes the complete $r$-partite graph with each part of size $s$. The graph $K_{1,t}$ is a \emph{$t$-star}. Denote by $\nu_t(G)$ the maximum number of vertex-disjoint $t$-stars in $G$.
For integers $p$ and $q$, we  let  $[p,q] = \{i \in \mathbb{Z} : p \le i \le q \}$.

\subsection{Main theorem and structure of the paper}\label{subsec:mainthm}

Our proof relies on the concept of \emph{bottle-graph}. A $k$-\emph{chromatic bottle-graph} $\mathcal{B}$ is a complete $k$-partite graph with partition sizes $(\sigma, \omega, \omega, \ldots, \omega)$, where $\sigma = \alpha \omega$ for some $0 < \alpha \le 1$. We refer to $\sigma$ and $\omega$ as the \emph{neck} and \emph{width} of $\mathcal{B}$, respectively. The critical chromatic number of $\mathcal{B}$ is
\begin{align*}
\chi_{cr}(\mathcal{B}) =(\chi(\mathcal{B}) - 1)\frac{|\mathcal{B}|}{|\mathcal{B}| - \sigma(\mathcal{B})}=(k-1)\frac{\sigma+(k-1)\omega}{\sigma+(k-1)\omega-\sigma}=k-1+\alpha.
\end{align*}
The vector
\[
\beta = \left(\frac{\alpha}{k-1+\alpha}, \frac{1}{k-1+\alpha}, \ldots, \frac{1}{k-1+\alpha}\right)
\]
is the \emph{color-vector} of $\mathcal{B}$.

For a $k$-chromatic $h$-vertex graph $H$ with smallest color-class size $\sigma = \sigma(H)$, we define the \emph{bottle-graph of $H$}, denoted $\mathcal{B}(H)$, as the smallest bottle-graph that contains an $H$-factor with the color-vector $\beta = (s, t, \ldots, t)$, where $s = \sigma/h$ and $t = (1-s)/(k-1)$. Then
\[\chi_{\mathrm{cr}}(\mathcal{B}(H))=k-1+\frac{s}{t}=\frac{k-1}{1-s}\ \mbox{and}\
\chi_{\mathrm{cr}}(H) =(k-1)\frac{h}{h-\sigma} =(k-1)\frac{1}{1-\frac{\sigma}{h}}=\frac{k-1}{1-s}.\]
To verify the existence of such a bottle-graph, let $\{\sigma, \sigma_1, \sigma_2, \ldots, \sigma_{k-1}\}$ be the sizes of the color-classes in a $k$-coloring of $H$. We can construct $\mathcal{B}(H)$ by taking $k-1$ vertex-disjoint copies of $H$ and cyclically shifting the color-classes. Specifically, in the $i$-th copy of $H$, we assign the color-classes of size $\sigma, \sigma_i, \sigma_{i+1}, \ldots, \sigma_{k-1}, \sigma_1, \ldots, \sigma_{i-1}$ to the partitions $1, 2, \ldots, k$ of $\mathcal{B}(H)$. Under this construction, the order of $\mathcal{B}(H)$ is at most $(k-1)h$. Therefore, it is sufficient to
prove Theorem~\ref{thm:main} when $H$ is a bottle-graph.

\begin{thm}\label{thm:maintheorem}
Let $H$ be a $k$-chromatic bottle-graph with neck $\sigma$ and width $\omega$. There exist constants $C= C(H)$ and $n_0$ such that if $G$ is a graph on $n \ge n_0$ vertices and
\begin{align}\label{eq:degree}
d(x) + d(y) \ge 2\left(1 - \frac{1}{\chi_{\mathrm{cr}}(H)}\right)n = 2\left(1 - \frac{1}{k-1} + \gamma\right)n
\end{align}
for all nonadjacent $x, y \in V(G)$, where $\alpha = \sigma/\omega$ and $\gamma = \frac{\alpha}{(k-1)(k-1+\alpha)}$, then $G$ contains an $H$-tiling covering all but at most $C$ vertices.
\end{thm}

The remainder of this paper is organized as follows. Section~\ref{sec:tool} introduces some technical tools. The proof of Theorem~\ref{thm:maintheorem} is presented in Sections~\ref{sec:proof}-\ref{sec:balancedcase}.

\section{Tools and clique cover}\label{sec:tool}

\subsection{Regularity Lemma and Blow-up Lemma}\label{subsec:lemma}
In the proof of our main theorem, we will use Szemer\'{e}di's Regularity Lemma~\cite{S1976} and the Blow-up Lemma of Koml\'{o}s, S\'{a}rk\"{o}zy, and Szemer\'{e}di~\cite{komlos1997}. We begin by introducing the necessary definitions.
Let $G=G(A,B)$ be a bipartite graph. The \emph{density} between $A$ and $B$ is defined as
\[
d(A,B) = \frac{e(A, B)}{|A||B|}.
\]
For $\varepsilon> 0$, the pair $(A, B)$ is \emph{$\varepsilon$-regular} if for all $X \subset A$ and $Y \subset B$ satisfying $|X| > \varepsilon|A|$, $|Y| > \varepsilon|B|$, we have $|d(X, Y) - d(A, B)| < \varepsilon$. Furthermore, for $d > 0$, the pair $(A, B)$ is \emph{$(\varepsilon, d)$-super-regular} if it is $\varepsilon$-regular and satisfies:
\[
d_G(a)> d|B|\quad \text{for all } a \in A, \quad \text{and} \quad d_G(b)>d|A|\quad \text{for all } b \in B.
\]

The Regularity Lemma provides a structural partition of a large and dense graph into a bounded number of clusters, where the edges between most pairs of clusters behave in a pseudo-random manner. We utilize the degree form of this lemma, which provides a structural partition
of the graph while discarding a small number of edges.

\begin{lemma}[{\bf Regularity Lemma-Degree Form}, Szemer\'{e}di~\cite{S1976}]\label{lemma:Regularity}
For every $\varepsilon > 0$, there exists $m=m(\varepsilon)$ and $M = M(\varepsilon)$ such that if $G $ is any graph and $d$ is any real number with $0\leq d\leq 1$, then there exists a partition of $V$ into $\ell + 1$ classes $V_0, V_1, \ldots, V_\ell$ and a  subgraph $G' \subseteq G$ such that
\begin{itemize}
    \item $m\le \ell \le M$,
    \item $|V_0| \le \varepsilon|V(G)|$,
     $|V_1| = |V_2| = \ldots = |V_{\ell}| = L \le \varepsilon|V(G)|$,
    \item $d_{G'}(v) > d_G(v) - (d + \varepsilon)|V(G)|$ for all $v \in V(G)$,
    \item $e(G'[V_i])=0$ for all $i\in [1,\ell]$,
    \item For all  $1 \le i < j \le \ell$, the pair $(V_i, V_j)$ is $\varepsilon$-regular in $G'$ and has density either $0$ or greater than $d$.
\end{itemize}
\end{lemma}

The sets $V_1,\ldots,V_{\ell}$ are called \emph{clusters}, $V_0$ is called the \emph{exceptional set}, and the
vertices in $V_0$ are \emph{exceptional vertices}.
The \emph{reduced graph} $R$ of $G$ is
the graph whose vertices are $V_1,\ldots,V_{\ell}$, where an edge exists between $V_i$ and $V_j$ if the pair $(V_i, V_j)$ is $\varepsilon$-regular with density greater than $d$ in $G'$.  The following lemma demonstrates that the Ore-type condition is almost preserved in the reduced graph.

\begin{lemma}[K\"{u}hn, Osthus, and Treglown~\cite{kuhn2009ore}]\label{lemma:Rdegree} Given a constant $C$, let $G$ be a graph such that $d_G(x)+ d_G(y)\geq C|V(G)|$ for all nonadjacent $x\neq y\in V (G)$. Let $R$ be the reduced graph obtained by applying Lemma~\ref{lemma:Regularity} with parameters $\varepsilon$ and $d$. Then for all nonadjacent $V_i\neq V_j\in V (R)$, $d_R(V_i)+d_R(V_j ) \geq (C- 2d-4\varepsilon)|V(R)|$.
\end{lemma}

The Blow-up Lemma guarantees that if a target graph $H$ with bounded maximum degree can be embedded into a "complete blow-up" graph of $R$, it can also be embedded into a super-regular "blow-up" graph $G$.

\begin{lemma}[{\bf Blow-up Lemma}, Koml\'{o}s, S\'{a}rk\"{o}zy, and Szemer\'{e}di~\cite{komlos1997}]\label{lemma:blowup}
Given a graph $R$ of order $r$ and positive parameters $\delta$ and $\Delta$, there exists an $\varepsilon = \varepsilon(\delta, \Delta, r)>0$ such that the following holds. Let $n_1, n_2, \ldots, n_r$ be arbitrary positive integers, and let us replace vertices $v_1,\ldots,v_{r}$ of $R$ with pairwise disjoint sets $V_1,\ldots,V_{r}$ of sizes $n_1,\ldots,n_r$ (blowing up). We construct two graphs on the same vertex set $V = \bigcup_{i\in [1,r]} V_i$.
 The first graph $R_b$ is obtained by replacing each edge $v_iv_j\in E(R)$ with the complete bipartite graph between the corresponding vertex sets $V_i$ and $V_j$.
A sparser graph $G$ is constructed by replacing each edge $v_iv_j\in E(R)$ arbitrarily with an $(\varepsilon, \delta)$-super-regular pair between $V_i$ and $V_j$.
If a graph $H$ with $\Delta(H) \le \Delta$ is embeddable into $R_b$, then $H$ is also embeddable into $G$.
\end{lemma}

\subsection{Clique cover}\label{subsec:cover}

In this section, we assume that $\mathcal{G}$ is a graph on $\ell$ vertices and for all nonadjacent $x, y \in V(\mathcal{G})$,
\begin{align}\label{eq:Orecondition}
d(x)+d(y) \ge 2\left(1 - \frac{1}{k-1} + \gamma - d-2\varepsilon\right)\ell.
 \end{align}

A $k$-\emph{clique-cover} $\Phi = \{\Phi_k, \Phi_{k-1}, \ldots, \Phi_1\}$ of $\mathcal{G}$ is a collection of pairwise vertex-disjoint cliques such that $V(\mathcal{G}) = \bigcup_{i=1}^k V(\Phi_i)$, where $\Phi_i$ denotes the set of cliques of order $i$ for each $i\in[1,k]$.
A $k$-clique-cover $\Phi$ is \emph{maximal} if the vector $(|\Phi_k|, |\Phi_{k-1}|, \ldots, |\Phi_1|)$ is lexicographically maximal among all $k$-clique-covers of $\mathcal{G}$. We denote a clique of order $i$ in the family $\Phi_i$ as $K^i$.
For $K^i \in \Phi_i$ and $K^j \in \Phi_j$ with $1 \le i\leq j\leq k$, we define the connectivity between these two cliques as follows:
\begin{itemize}
    \item $K^{i}$ is \emph{well-connected} to  $K^{j}$, denoted by $K^i \hookrightarrow K^j$, if $e(\{v\}, K^j) = j-1$ for all $v \in K^i$.
    \item $K^{i}$ is  \emph{over-connected} to  $K^{j}$, denoted by $K^i \overconnect K^j$, if $e(K^i, K^j) \ge i(j-1)$ but $K^i \not\hookrightarrow K^j$.
    \item $K^{i}$ is  \emph{under-connected} to  $K^{j}$, denoted by $K^i \underconnect K^j$, if $e(K^i, K^j) < i(j-1)$.
\end{itemize}

The following properties of maximal clique-covers are essential to our proof. Proposition~\ref{prop:connectivity} was established by Shokoufandeh and Zhao~\cite{zhao2003}.

\begin{prop}[Shokoufandeh and Zhao~\cite{zhao2003}]\label{prop:connectivity} Let $1\leq i\leq j< k$. Then
\begin{enumerate}[{\rm(i)}]
    \item $e(K^i, K^j)\leq i(j-1)$.
    In particular, if $e(K^i, K^j) = i(j-1)$, then $K^i \hookrightarrow K^j$.

    \item For any $K^k \in \Phi_k$ and any set of $k$ cliques $\{K_1^i, \ldots, K_k^i\} \subseteq \Phi_i$, we have $e(\bigcup_{p=1}^k K_{p}^i, K^k) \le ik(k-1)$.

    \item If $K^j \in \Phi_j$ and $K_1^i, K_2^i \in \Phi_i$ satisfy $e(K_1^i, K^j) = e(K_2^i, K^j) = i(j-1)$, then $K_1^i \hookrightarrow K^j$ and $K_2^i \hookrightarrow K^j$.
Furthermore, $K^j$ can be partitioned as $K^j = A_i(K^j) \cup B_i(K^j)$ such that  $|A_i(K^j)| = i$,  $|B_i(K^j)| = j-i$, and for all $v \in A_i(K^j)$, for all $u \in B_i(K^j)$, we have $d(v, K_1^i) = d(v, K_2^i) = i-1$ and
  $d(u, K_1^i) =d(u, K_2^i) = i$.
  Finally, let
\[S_i^j=\{K^j\in \Phi_j: e(K_1^i, K^j) = e(K_2^i, K^j) = i(j-1)\}\ \text{and}\ A_i^j=\{A_i(K^j):K^j\in S_i^j\}.\]
Then $\mathcal{G}[A_i^j]$ is an $i$-partite subgraph with $|S_i^j|$ vertices in each color-class. Moreover, $\mathcal{G}[ A_i^i\cup A_i^{i+1}\cup\ldots\cup A_i^k]$ is an $i$-partite graph with $\sum_{j=i}^k |S_i^j|$
vertices in each color-class.
\end{enumerate}

\end{prop}

\begin{prop}\label{prop:lowerbound}
If $K_1,\ldots,K_{p}$ are $p\geq 2$ cliques in $\Phi_{i}$ for some $i\in [1,k-1]$, then \[e\left(\bigcup\limits_{j=1}^{p} K_{j}, \mathcal{G}\right)\geq p i\left(1 - \frac{1}{k-1} + \gamma - d-2\varepsilon\right)\ell.\]
\end{prop}

\begin{proof}
We prove the proposition by first showing that the vertices of any two cliques $K,K'\in \Phi_i$ can be paired as in the following claim.

\begin{claim}\label{claim:PMKK'}
For any two distinct cliques $K,K'\in \Phi_i$, the complement of the bipartite graph $\mathcal{G}[K,K']$ contains a perfect matching between $K$ and $K'$.
\end{claim}

\begin{proof}
Suppose not. By Hall's theorem, there exists a set $S\subseteq V(K)$ such that $|N_{\overline{\mathcal{G}}}(S,V(K'))|<|S|$, where $N_{\overline{\mathcal{G}}}(S,V(K'))$ denotes the set of vertices in $K'$ that are non-adjacent in $\mathcal{G}$ to any vertex of $S$. Let $T=V(K')\setminus N_{\overline{\mathcal{G}}}(S,V(K'))$.
Then any vertex of $S$ is adjacent in $\mathcal{G}$ to every vertex of $T$. Since both $K$ and $K'$
are cliques, the set $S\cup T$ is a clique in $\mathcal{G}$. Moreover,
\[
|S\cup T|=|S|+|T|=|S|+i-|N_{\overline{\mathcal{G}}}(S,V(K'))|>i.
\]
Since $i\le k-1$, choose a subset $Q\subseteq S\cup T$ with $|Q|=i+1$.  The set $Q$ is a clique.  Replacing $K$ and $K'$ by the $(i+1)$-clique $Q$, and covering the remaining vertices of $(K\cup K')\setminus Q$ by singleton cliques if necessary, gives a $k$-clique-cover. Its vector is lexicographically larger than that of $\Phi$, because it has one additional clique of order $i+1$ while no clique of order larger than $i+1$ is changed. This contradicts the maximality of $\Phi$. Hence the claim follows.
\end{proof}

Denote by $t=\left(1-\frac{1}{k-1}+\gamma-d-2\varepsilon\right)\ell$ and $K_{p+1}=K_1$. By Claim~\ref{claim:PMKK'}, for each $j\in [1,p]$, there exists a perfect matching in the complement of $\mathcal{G}[K_j,K_{j+1}]$. Hence we may pair the vertices of $K_j$ and $K_{j+1}$ as $\{(x_{j,s},y_{j,s}):s\in [1,i]\}$,
where $x_{j,s}\in V(K_j)$, $y_{j,s}\in V(K_{j+1})$, and $x_{j,s}y_{j,s}\notin E(\mathcal{G})$
for all $s\in [1,i]$.

By the Ore-type condition~\eqref{eq:Orecondition}, for all $j\in [1,p]$ and $s\in [1,i]$, we have
$d_{\mathcal{G}}(x_{j,s})+d_{\mathcal{G}}(y_{j,s})\ge 2t$.
Then
\[e(K_j\cup K_{j+1},\mathcal{G})=\sum_{u\in V(K_j)}d_{\mathcal{G}}(u)+\sum_{u\in V(K_{j+1})}d_{\mathcal{G}}(u)
=\sum_{s=1}^{i}(d_{\mathcal{G}}(x_{j,s})+d_{\mathcal{G}}(y_{j,s}))\ge 2it.\]
It follows that $\sum_{j=1}^{p} e(K_j\cup K_{j+1},\mathcal G)=2\sum_{u\in \bigcup_{j=1}^{p}V(K_j)}d_{\mathcal G}(u)\ge 2pit$.
Therefore,
\[
e\left(\bigcup_{j=1}^{p}K_j,\mathcal{G}\right)=\sum_{u\in \bigcup_{j=1}^{p}V(K_j)}d_{\mathcal{G}}(u) \ge pit.
\]
\end{proof}

For $i\in [1,k]$, let  $\varphi_i = |\Phi_i|/\ell$ be the normalized size of $\Phi_i$. Then
\begin{equation}\label{eq:sum1}
\sum_{i=1}^k i\varphi_i = 1.
\end{equation}

Let $i_0 = \min\left\{i\in [1,k]: |\Phi_i| \ge k\right\}$. We proceed under the following assumption, which is required for the proof of the main result.

\begin{assum}\label{assume}
Assume that $i_0 < k$ and $\varphi_i = 0$ for $i < i_0$.
\end{assum}
The following proposition establishes a lower bound on $\varphi_{k} $ under the Ore-type condition~\eqref{eq:Orecondition}.

\begin{prop}\label{prop:varphik}
$
\varphi_k \geq \sum\limits_{i=2}^{k-i_0} (i-1)\varphi_{k-i} + (k-1)\gamma - (k-1)(d+2\varepsilon).
$
\end{prop}

\begin{proof}
Let $K_1, \ldots, K_k \in \Phi_{i_0}$ be $k$ distinct cliques of order $i_0$.  By Proposition~\ref{prop:connectivity}(i) and (ii), we have
\[
e\left(\bigcup_{p=1}^k K_{p}, \Phi_j\right)\le k i_0 (j-1)\varphi_j \ell\ \text{for\ all}\ j\in[i_0,k-1],\quad \text{and}\quad
e\left(\bigcup_{p=1}^k K_{p}, \Phi_k\right) \le k i_0 (k-1)\varphi_k \ell.
\]
By Assumption~\ref{assume}, $e\left(\bigcup_{p=1}^k K_{p}, \Phi_j\right)=0$ for all $j\in[1,i_0-1]$. Combining the above  inequalities, we get
\begin{equation}\label{eq:upperbound}
e\left(\bigcup_{p=1}^k K_{p}, \mathcal{G}\right)\le k i_0 \ell \sum_{j=i_0}^k (j-1)\varphi_j.
\end{equation}

Combining~\eqref{eq:upperbound} with Proposition~\ref{prop:lowerbound}, we obtain
\[
k i_0 \left(1 - \frac{1}{k-1} + \gamma - d -2\varepsilon \right)\ell \le e\left(\bigcup_{p=1}^k K_{p}, \mathcal{G}\right)\leq
k i_0 \ell \sum_{j=i_0}^k (j-1)\varphi_j
\]
It implies that
\begin{align}
 & \quad 1 - \frac{1}{k-1} + \gamma - d-2\varepsilon
\le \sum_{j=i_0}^k (j-1)\varphi_j  =\sum_{j=i_0}^k \left( \frac{k-2}{k-1}j + \frac{j-k+1}{k-1} \right) \varphi_j\nonumber \\
&= \frac{k-2}{k-1} \sum_{j=i_0}^k j\varphi_j + \frac{1}{k-1} \sum_{j=i_0}^k (j-k+1)\varphi_j\nonumber \\
&{\overset{\eqref{eq:sum1}}=} \frac{k-2}{k-1}\cdot 1 + \frac{1}{k-1} \left((k-k+1) \varphi_k + \sum_{j=i_0}^{k-1} (j-k+1)\varphi_j \right)\nonumber\\
&= \frac{k-2}{k-1} + \frac{1}{k-1} \varphi_k - \frac{1}{k-1} \sum_{j=i_0}^{k-2} (k-1-j)\varphi_j \nonumber\\
&{\overset{i=k-j}=} \frac{1}{k-1} \left( \varphi_k + k-2 - \sum_{i=2}^{k-i_0} (i-1)\varphi_{k-i} \right). \label{eq:process}
\end{align}

So $\varphi_k \geq \sum\limits_{i=2}^{k-i_0} (i-1)\varphi_{k-i} + (k-1)\gamma - (k-1)(d+2\varepsilon)$.
\end{proof}

By Proposition~\ref{prop:varphik}, we can let
\begin{equation}\label{eq:varphik}
\varphi_k = \sum_{i=2}^{k-i_0} (i-1)\varphi_{k-i} + (k-1)\gamma + s,
\end{equation}
where $s \ge -(k-1)(d+2\varepsilon) $.

Consider a clique $K \in \Phi_{k-i}$ for some $i \in [1, k-i_0]$ that is not over-connected to any $k$-clique. Let $\Lambda(K)$ denote the set of all $k$-cliques that are well-connected to $K$, and let $\lambda_i$ be the minimum value of $|\Lambda(K)|/\ell$ among all such cliques $K\in \Phi_{k-i}$, with at most one exception.
The next proposition provides a lower bound on $\lambda_i$. In applications of Proposition~\ref{prop:lambdai}, the possible exceptional clique may be absorbed into the exceptional vertex set. This increases that set by at most $k-i\le k$ vertices, and hence has no significant effect on the final bound.

\begin{prop}\label{prop:lambdai}
\begin{align*}
\lambda_i \ge (k-1)\gamma + \frac{i-1}{k-1}s + \sum_{2\le j\le i} (j-1)\varphi_{k-j} + (i-1)\sum_{j>i} \varphi_{k-j} - (k-i)(d+2\varepsilon).
\end{align*}
\end{prop}

\begin{proof}
Let $m =\frac{1}{\ell} \left|\{K^k \in \Phi_k \colon e(K, K^k) < (k-1)(k-i)\}\right|$. We prove the proposition by bounding $e(K, \mathcal{G})/\ell$ from below and above. First we give the upper bound. By Proposition~\ref{prop:connectivity}(i) and (ii),
\begin{align}
\frac{e(K, \mathcal{G})}{\ell} &=\frac{e(K,\Phi_k)+e(K,\cup_{j=1}^{k-i-1}\Phi_{j})+e(K,\cup_{j=k-i}^{k-1}\Phi_{j})}{\ell}\nonumber\\
&\leq (k-1)(k-i)\varphi_k - m +  \sum_{j=1}^{k-i-1} j(k-i-1)\varphi_{j}+\sum_{j=k-i}^{k-1} (k-i)(j-1)\varphi_{j} \nonumber\\
&\leq (k-1)(k-i)\varphi_k - m+\sum_{j>i} (k-i-1)(k-j)\varphi_{k-j} +  \sum_{j=1}^{i} (k-i)(k-j-1)\varphi_{k-j}\nonumber \\
& = (k-i)\left((k-1)\varphi_k + \sum_{j=1}^{i}(k-j-1)\varphi_{k-j} + \sum_{j>i} (k-j-1)\varphi_{k-j}\right)
+ \sum_{j>i} (j-i)\varphi_{k-j} - m\nonumber\\
&= (k-i)\sum\limits_{j=i_0}^{k}(j-1)\varphi_{j}+ \sum_{j>i} (j-i)\varphi_{k-j} - m. \label{eq:upper-K}
\end{align}

Now we bound  $e(K, \mathcal{G})/\ell$ from below. By the proof of Proposition~\ref{prop:lowerbound}, for any two distinct cliques $K_1,K_2\in \Phi_{k-i}$, we have
\[
e(K_1\cup K_2,\mathcal{G}) \ge 2(k-i)\left(1-\frac{1}{k-1}+\gamma-d-2\varepsilon\right)\ell .
\]
By the pigeonhole principle, at least one of $K_1$ and $K_2$ satisfies
\[
e(K_j,\mathcal{G})\ge(k-i)\left(1-\frac{1}{k-1}+\gamma-d-2\varepsilon\right)\ell, j\in[1,2].
\]
Consequently, at most one clique in $\Phi_{k-i}$ can fail this inequality. Thus, except possibly for one clique in $\Phi_{k-i}$, we may assume that
\begin{equation}\label{eq:lower-K}
\frac{e(K,\mathcal{G})}{\ell} \ge(k-i)\left(1-\frac{1}{k-1}+\gamma-d-2\varepsilon\right).
\end{equation}

Combining~\eqref{eq:upper-K} and~\eqref{eq:lower-K}, we obtain
\begin{align*}
(k-i)\left(1 - \frac{1}{k-1} + \gamma - d -2\varepsilon\right)\le (k-i)\sum\limits_{j=i_0}^{k}(j-1)\varphi_{j}+ \sum_{j>i} (j-i)\varphi_{k-j} - m.
\end{align*}
Thus
\begin{align*}
m &\leq (k-i)\sum\limits_{j=i_0}^{k}(j-1)\varphi_{j}+ \sum_{j>i} (j-i)\varphi_{k-j}-(k-i)\left(1 - \frac{1}{k-1} + \gamma - d -2\varepsilon\right) \\
&{\overset{\eqref{eq:process}}\leq} \frac{k-i}{k-1}\left(\varphi_k + k-2 - \sum_{j=2}^{k-i_0}(j-1) \varphi_{k-j}\right)+ \sum_{j>i} (j-i)\varphi_{k-j}
-(k-i)\left(1 - \frac{1}{k-1} + \gamma - d -2\varepsilon\right)\\
&{\overset{\eqref{eq:varphik}}\leq} \frac{k-i}{k-1}\left(\sum_{j=2}^{k-i_0} (j-1)\varphi_{k-j} + (k-1)\gamma + s + k-2 - \sum_{j=2}^{k-i_0}(j-1) \varphi_{k-j}\right)+ \sum_{j>i} (j-i)\varphi_{k-j}\\
&\quad-(k-i)\left(1 - \frac{1}{k-1} + \gamma - d -2\varepsilon\right)\\
&= \frac{k-i}{k-1}s + \sum_{j>i} (j-i)\varphi_{k-j} + (d +2\varepsilon)(k-i).
\end{align*}

By the definition of $K$, $K$ is not over-connected to any $k$-clique, so every $k$-clique $K^k\in \Phi_k$ satisfying
$e(K,K^k)\ge (k-i)(k-1)$ is in fact well-connected to $K$. Therefore,
\begin{align*}
\frac{|\Lambda(K)|}{\ell}&\geq \varphi_{k}-m {\overset{\eqref{eq:varphik}}\geq}\sum_{j=2}^{k-i_0} (j-1)\varphi_{k-j} + (k-1)\gamma + s-\frac{k-i}{k-1}s - \sum_{j>i} (j-i)\varphi_{k-j}- (d +2\varepsilon)(k-i)\\
&= (k-1)\gamma + \frac{i-1}{k-1}s + \sum_{2\le j\le i} (j-1)\varphi_{k-j} + (i-1)\sum_{j>i} \varphi_{k-j} - (k-i)(d+2\varepsilon).
\end{align*}
Taking the minimum over all eligible cliques $K\in \Phi_{k-i}$, with possibly one exceptional clique described above, gives the desired lower bound on $\lambda_i$. We complete the proof of Proposition~\ref{prop:lambdai}.
\end{proof}

\section{Proof sketch of Theorem~\ref{thm:maintheorem}}\label{sec:proof}

In this section, we outline the proof of Theorem~\ref{thm:maintheorem}. Throughout the proof, we assume that $n$ is sufficiently large. Recall that $\sigma\leq \omega$.
If $\sigma<\omega$, then we assume that
\begin{align}\label{eq:parameter}
  0<  \varepsilon \ll d \ll \mu \ll \min\{\alpha, 1 - \alpha\}.
\end{align}
If $\sigma=\omega$, then $\alpha=1$ and we  instead assume that
\begin{align}\label{eq:parameter2}
  0<  \varepsilon \ll d \ll \mu \ll 1.
\end{align}

Now we sketch the proof of Theorem~\ref{thm:maintheorem}. The proof begins by applying  Lemma~\ref{lemma:Regularity} to $G$ with parameters $\varepsilon$ and $d$. This yields a partition of $V(G)$ into an exceptional set $V_0$ and clusters $V_1, \ldots, V_\ell$ of equal size $L$. For $i \ge 1$, we assume $L$ is divisible by specific integers to be determined later. If necessary, this can be achieved by moving only a constant number of vertices from each cluster $V_i$ into $V_0$.

Let $R$ be the reduced graph of $G$ and let $G''=G- V_{0}$. By Lemma~\ref{lemma:Rdegree},
\begin{align}\label{eq:oreconditionR}
d_R(V_i)+d_R(V_j ) \geq 2\left(1-\frac{1}{k-1}+\gamma- d-2\varepsilon\right)\ell\ \text{for\ all\ nonadjacent}\ V_i, V_j\in V (R).
\end{align}

We next describe the local structure arising from the clique-cover of the reduced graph. An \emph{$\varepsilon$-regular $r$-clique}, denoted by $\mathcal{R}_{\varepsilon}(n_1,\ldots,n_r)$,
is an $r$-partite graph with parts of sizes $n_1,\ldots,n_r$, such that every pair of parts forms an $\varepsilon$-regular pair. We write $\mathcal{R}_{\varepsilon,r}^{a}(t)$  for an $\varepsilon$-regular $r$-clique with  $n_i = t$ for all $i\in [1, r-1]$ and $n_r=at$,
where $0<a\leq 1$. We call $\mathcal{R}_{\varepsilon,r}^{1}(t)$ \emph{balanced}, and $\mathcal{R}_{\varepsilon,r}^{a}(t)$ \emph{unbalanced} when $a<1$.

Let $\Phi = \{\Phi_k, \Phi_{k-1}, \ldots, \Phi_1\}$ be a maximal clique-cover of the reduced graph $R$. This clique-cover naturally induces a \emph{cluster-clique-cover} $\Psi = \{\Psi_k, \Psi_{k-1}, \ldots, \Psi_1\}$ of $G''$, where each $\Psi_i$ consists of  pairwise disjoint $\varepsilon$-regular $i$-cliques for $i\in [1,k]$.
We denote an $\varepsilon$-regular $i$-clique by $\EuScript{K}^i$. All terminology introduced in Subsection~\ref{subsec:cover} applies to this cluster-clique-cover as well.

Let  $i_0 = \min\left\{i\in [1,k]: |\Phi_i| \ge k\right\}$.
We may assume that $i_0 < k$. Indeed, if $i_0=k$, then there are fewer than $k$ cliques in each of $\Phi_1,\ldots,\Phi_{k-1}$. Hence, by moving all vertices of $G$ contained in the clusters corresponding to these smaller cliques into $V_0$, the remaining graph $G''$ admits an $\mathcal{R}_{\varepsilon,k}^{1}(L)$-covering, while still satisfying $|V_0|=O(\varepsilon n)$. Similarly, we may assume that $\varphi_i = 0$ for all $i < i_0$, since the vertices belonging in clusters corresponding to such smaller cliques may also be absorbed into the
exceptional set $V_0$. Thus, we can apply Proposition~\ref{prop:varphik} to $R$ and derive equation~\eqref{eq:varphik}. For notational simplicity, we assume $i_0 = 1$.

The proof of Theorem~\ref{thm:maintheorem} is divided into two main cases: $\sigma<\omega$ and $\sigma=\omega$. In the first case, $\sigma<\omega$, we further distinguish two subcases: the \emph{general subcase}, where $s\geq \mu$, and the \emph{extremal subcase}, where $s<\mu$.

In the general subcase (Section~\ref{sec:general}), we show that $G''$ can be tiled by copies of $\mathcal{R}_{\varepsilon,k}^{\alpha'}(L_1)$, where $\alpha < \alpha' \le 1$ and $L_1 = CL$ for some constant $0 < C \le 1$.
Given one such $H$-tiling, the vertices of $V_0$ will be inserted into some clusters
of appropriate $\varepsilon$-regular $k$-cliques such that after we remove copies of $H$ containing new
vertices, the remaining parts of the cliques still contain almost perfect $H$-tilings. We further use the connections among different $\varepsilon$-regular $k$-cliques to reduce the total number
of uncovered vertices to a constant depending only on $H$.

In the extremal subcase (Section~\ref{sec:excase}), we prove that $V(G)$ can be partitioned into two sets $A$ and $B$. Here, $A$ is the union of several almost-independent sets $U_1, \ldots, U_t$ with $|U_i| \approx \frac{n}{k-1+\alpha} $. There are two cases for the set $B$:
\begin{enumerate}[{\rm(i)}]
    \item $B$ can be partitioned into almost-independent sets $U_{t+1}, \ldots, U_k$, where $|U_i| \approx \frac{n}{k-1+\alpha}$ for $i < k$ and $|U_k| \approx\frac{\alpha n}{k-1+\alpha}$;
    \item $B$ can be almost tiled by copies of $K(\omega, \ldots,\omega,\sigma)$.
\end{enumerate}
In either case, we show that all but a constant number of vertices can be covered by
an $H$-tiling.

Finally, suppose that $\sigma = \omega$. Then $H$ is the balanced complete $k$-partite graph $K_k(\omega)$. In this case, the vertex-shifting mechanism developed for the case $\sigma<\omega$ is no longer applicable, since that mechanism relies on exchanging a $\sigma$-class for an $\omega$-class and thereby shifting $\omega-\sigma$ unused vertices between clusters. When $\sigma=\omega$, this difference is zero. Nevertheless, the balanced case has a more flexible structure: the chromatic bottleneck present in the unbalanced case disappears. We therefore adapt the method used for $\sigma<\omega$ and give the details in Section~\ref{sec:balancedcase}.

\section{The general subcase when $s\ge \mu$}\label{sec:general}

As explained in Section~\ref{sec:proof}, we assume throughout this section that \(i_0=1\). Thus, by~\eqref{eq:varphik}, $s = \varphi_k - \sum_{i=2}^{k-1} (i-1)\varphi_{k-i} - (k - 1)\gamma$. In this section, we consider the general subcase $s\ge \mu$. In Subsection~\ref{subsec:dlemma}, we prove that $G''$ has an $\mathcal{R}_k^{\alpha'}(L_1)$-factor for some $\alpha<\alpha'\le 1$ and $L_1= CL$ for some constant $0<C\le 1$. Based on this statement, we prove that $G$ has an $H$-tiling that covers all but at most a constant number of vertices in Subsection~\ref{subsec:exvtx}.

\subsection{The decomposition lemma}\label{subsec:dlemma}

Define
\begin{align}\label{eq:alpha'}
    \alpha' = \alpha + \frac{\alpha(1-\alpha)}{k^2}\mu.
\end{align}
Without loss of generality, we assume $\mu =1/N$ for some integer $N$. Recall that $\alpha = \sigma/\omega$, so \[\alpha'=\frac{\sigma}{\omega}+\frac{\frac{\sigma}{\omega}(1-\frac{\sigma}{\omega})}{k^2}\cdot \frac{1}{N}=\frac{\omega \sigma k^2 N+\sigma\omega-\sigma^2}{\omega^2k^2N}\] is a rational number. Set $\alpha' = p/q$, where $p$ and $q$ are two integers with no
common factors. As $\omega \sigma k^2 N+\sigma\omega-\sigma^2-\omega^2 k^2 N=(\omega-\sigma)(\sigma-\omega k^2 N)< 0$, we can assume that $p< q \le \omega^2k^2N =\omega^2k^2/\mu$. The main result in this subsection is the following decomposition lemma.

\begin{lemma}[{\bf Decomposition Lemma}]\label{lem:DS}
If the reduced graph $R$ satisfies $s\ge \mu$, that is,
\[
\varphi_k \ge \sum_{i=2}^{k-1} (i-1)\varphi_{k-i} + (k-1)\gamma + \mu,
\]
then $G''$ can be decomposed into vertex-disjoint copies of $\mathcal{R}_k^{\alpha'}(L_1)$ for some $\alpha' = \alpha + \frac{\alpha(1-\alpha)}{k^2}\mu$ and $L_1 \ge c\mu L$.
\end{lemma}

To prove Lemma~\ref{lem:DS}, we introduce some necessary terminology and an algorithm.
For a given $i\in [1, k-1]$, consider an $\varepsilon$-regular $(k-i)$-clique $\mathcal{K} \in \Psi_{k-i}$, with corresponding clique $K \in \Phi_{k-i}$. An $\varepsilon$-regular $k$-clique $\mathcal{K}^k$ is \emph{good} for $\mathcal{K}$ if its corresponding clique $K^k \in \Phi_k$ is either well-connected or over-connected to $K$; otherwise, $\mathcal{K}^k$ is \emph{bad} for $\mathcal{K}$. We call $\mathcal{K}$ \emph{typical} if $|\{K^k \in \Phi_k : K \overconnect K^k\}| < c_t$ for a constant $c_t = \left\lceil \frac{k-1+\alpha'}{1-\alpha'} \right\rceil$.  Otherwise, $\mathcal{K}$ is \emph{atypical}. To formulate the decomposition algorithm for $G''$, we rely on the following lemma, which states that sufficiently large restrictions of regular pairs remain regular.

\begin{lemma}[{\bf Slicing Lemma}]\label{lem:SL}
    Let $V_i$ and $V_j$, $1 \le i < j \le \ell$, be two clusters in $G''$ corresponding to the two endvertices of an edge $e$ in the reduced graph $R$. Partition $V_i$ and $V_j$ into $s$ and $t$ subclusters $\{V_i^1, \ldots, V_i^s\}$ and $\{V_j^1, \ldots, V_j^t\}$, respectively, such that the size of the smallest subclusters is $cL$ for some $\varepsilon \ll c \le 1$. Then the pairs $(V_i^p, V_j^q)$ for $ p\in [1, s]$ and $q\in [1, t]$ are $\varepsilon'$-regular pairs with $\varepsilon' =\min\{\varepsilon/2, \varepsilon/c\}$. In particular, if $s = t$, then $(V_i^p, V_j^p)$ are $s$ disjoint $\varepsilon'$-regular pairs for $p\in [1, s]$.
\end{lemma}

When each cluster in an $\varepsilon$-regular clique $\mathcal{K}$ is evenly partitioned into $s$ parts, we obtain
$s$ new $\varepsilon'$-regular cliques consisting of these smaller clusters by Lemma~\ref{lem:SL}. We call this operation an $s$-\emph{partition} of $\mathcal{K}$. Although $\varepsilon$ has changed to $\varepsilon'$ in the new regular pairs, for notational simplicity, we always use $\varepsilon$ as the parameter. Naturally, we call a new
clique \emph{good} (or \emph{bad}) if it is derived from a good (or bad) larger clique $\mathcal{K}$.

Since $L$ has been chosen in a way that it can be divided by particular values, we assume
that all divisions used in the following algorithm are integral.

\textbf{Decomposition Algorithm for $G''$:}
\begin{enumerate}[{\bf Step 1}]
\item \label{step1} Perform a $(q-p)$-partition of every clique in $\Psi_k$. Denote the resulting family of $k$-cliques by $\Psi_k'$.

        We use $L'$ to denote the size of each new subcluster, that is, $L' = \frac{L}{q-p}>\frac{L}{q}\geq\frac{\mu L}{\omega^2k^2}$.

\item \label{step2} For each $i\in [1,k-1]$, convert every typical clique in $\Psi_{k-i}$ into copies of unbalanced $\varepsilon$-regular $k$-cliques.

Let $\mathcal K\in \Psi_{k-i}$ be a typical $\varepsilon$-regular $(k-i)$-clique. Each cluster of $\mathcal K$ has size $L$. We eliminate $\mathcal K$ using suitable well-connected $k$-cliques from $\Psi_k'$, whose clusters have size $L'=\frac{L}{q-p}$.

Let $\mathcal K^k\in \Psi_k'$ be a well-connected $k$-clique for $\mathcal K$. By Proposition~\ref{prop:connectivity}(iii), the clusters of $\mathcal K^k$ can be partitioned into two parts
$A(\mathcal K^k)\cup B(\mathcal K^k)$, where $|A(\mathcal K^k)|=k-i$ and $|B(\mathcal K^k)|=i$.
The clusters in $A(\mathcal K^k)$ correspond to the clusters of $\mathcal K$, while the clusters in $B(\mathcal K^k)$ will be split.
For each cluster $V\in B(\mathcal K^k)$, partition $V$ into two parts $V=V^{\rm large}\cup V^{\rm small}$,
where $|V^{\rm large}|=\frac{i-1+\alpha'}{i}L'$ and $|V^{\rm small}|=\frac{1-\alpha'}{i}L'$.
We now deal with the large parts and the small parts separately.

\textbf{Large parts.}
Let $t_1=\frac{L'}{i}$. Each cluster in $A(\mathcal K^k)$ has size $L'=it_1$,
and each large part $V^{\rm large}$, where $V\in B(\mathcal K^k)$, has size $\frac{i-1+\alpha'}{i}L'=(i-1+\alpha')t_1$.
Hence the clusters in $A(\mathcal K^k)$, together with the large parts of the clusters in $B(\mathcal K^k)$, form an $\varepsilon$-regular $k$-partite graph with part sizes
\[\underbrace{it_1,\ldots,it_1}_{k-i},
\underbrace{(i-1+\alpha')t_1,\ldots,(i-1+\alpha')t_1}_{i}.\]
This graph can be decomposed into $i$ copies of $\mathcal R_{\varepsilon,k}^{\alpha'}(t_1)=\mathcal R_{\varepsilon,k}^{\alpha'}\left(\frac{L'}{i}\right)$.

\textbf{Small parts.}
Let $t_2=\frac{(1-\alpha')L'}{i(i-1+\alpha')}$.
Then each small part $V^{\rm small}$, where $V\in B(\mathcal K^k)$, has size $\frac{1-\alpha'}{i}L'=(i-1+\alpha')t_2$.
From each cluster of $\mathcal K$, take a subset of size $it_2=\frac{(1-\alpha')L'}{i-1+\alpha'}$.
Together with the small parts of the clusters in $B(\mathcal K^k)$, these chosen subsets form an $\varepsilon$-regular $k$-partite graph with part sizes
\[\underbrace{it_2,\ldots,it_2}_{k-i},
\underbrace{(i-1+\alpha')t_2,\ldots,(i-1+\alpha')t_2}_{i}.\]
This graph can be decomposed into $i$ copies of $\mathcal R_{\varepsilon,k}^{\alpha'}(t_2)=\mathcal R_{\varepsilon,k}^{\alpha'}\left(\frac{(1-\alpha')L'}{i(i-1+\alpha')}\right)$.

Thus, using one well-connected $k$-clique $\mathcal K^k\in\Psi_k'$, we obtain copies of
$\mathcal R_{\varepsilon,k}^{\alpha'}\left(\frac{L'}{i}\right)$ and $\mathcal R_{\varepsilon,k}^{\alpha'}\left(
\frac{(1-\alpha')L'}{i(i-1+\alpha')}\right)$, and we remove $\frac{(1-\alpha')L'}{i-1+\alpha'}$ vertices from each cluster of $\mathcal K$.
Since each cluster of $\mathcal K$ has size $L$, the number of well-connected $k$-cliques needed to eliminate $\mathcal K$ is $\frac{L}{\frac{(1-\alpha')L'}{i-1+\alpha'}}=\frac{(i-1+\alpha')L}{(1-\alpha')L'}$.
Repeating the above procedure with this many well-connected $k$-cliques from $\Psi_k'$, we eliminate $\mathcal K$ completely.

\item \label{step3}  For each $i\in[1,k-1]$, convert every atypical clique in $\Psi_{k-i}$ into copies of unbalanced $\varepsilon$-regular $k$-cliques.

The procedure is the same as in Step~\ref{step2}, except that we now use $k$-cliques which are over-connected to the atypical clique $\mathcal K$. If $K^k$ is over-connected to $K\in\Phi_{k-i}$, then at least $i$ vertices of $K^k$ are adjacent to every vertex of $K$. Indeed, if $a$ vertices of $K^k$ are adjacent to all vertices of $K$, then  \[e(K,K^k)\le a(k-i)+(k-a)(k-i-1)=k(k-i-1)+a. \]
    Since $e(K,K^k)\ge (k-i)(k-1)$, it follows that $a\ge i$. We choose $i$ such clusters to form $B(\mathcal K^k)$, and let the remaining clusters form $A(\mathcal K^k)$.

    \item \label{step4} Partition the remaining $k$-cliques in $\Psi_k'$ and all (different-size) unbalanced $\varepsilon$-regular $k$-cliques into copies of $\mathcal{R}_k^{\alpha'}(L_1)$.

    Here $L_1 = c_l L$, where the constant $c_l>0$ is  a multiple of $\mu$.
\end{enumerate}

\begin{proof}[Proof of Lemma~\ref{lem:DS}]
Note that Lemma~~\ref{lem:DS} holds if the Decomposition Algorithm given above is correct. So we prove the correctness of the algorithm. For Step~\ref{step3}, by Proposition~\ref{prop:connectivity}(ii), any $k$-clique that is over-connected to $\mathcal{K}$ does not participate in the conversion of any other (typical or atypical) clique. Since each $\mathcal{K} \in \Psi_{k-i}$ requires $\frac{(i-1+\alpha')L}{(1-\alpha')L'}$ good $k$-cliques in $\Psi_k'$, the total number of over-connected $k$-cliques is at least
\[
c_t \frac{L}{L'} = \left\lceil \frac{k-1+\alpha'}{1-\alpha'} \right\rceil \frac{L}{L'} \ge \frac{(i-1+\alpha')L}{(1-\alpha')L'}, \quad \text{for } i < k.
\]
Thus Step~\ref{step3} can be performed for every atypical clique.

During Step~\ref{step2}, we considered all typical cliques in $\Psi_{k-1}, \ldots, \Psi_1$ sequentially. For a typical clique $\mathcal{K} \in \Psi_{k-i}$, we can ignore the existence of its over-connected $k$-cliques in
our computation, since $c_t \ll \ell$. Thus, the term $\lambda_i \ell$ in Proposition~\ref{prop:lambdai} gives a lower bound on the number of well-connected $k$-cliques in $\Psi_k$. It follows that the new family $\Psi_k'$ contains at least $\lambda_i \ell L / L'$ well-connected $k$-cliques for $\mathcal{K}$. We need to show
\[\frac{\lambda_i \ell L}{L'}>\sum_{l=1}^i \frac{(l-1+\alpha')L}{(1-\alpha')L'} \varphi_{k-l} \ell.\]
Equivalently, we need to prove
\begin{align}\label{eq:Ii}
  \lambda_i - \sum_{l=1}^i \frac{l-1+\alpha'}{1-\alpha'} \varphi_{k-l} > 0.
\end{align}
Firstly, we compute
\begin{align*}
 I_i =\lambda_i - \sum_{l=1}^i \frac{l-1+\alpha}{1-\alpha} \varphi_{k-l}.
\end{align*}
By Proposition~\ref{prop:lambdai}, $\lambda_1\geq (k-1)\gamma$.
By~\eqref{eq:sum1}, $\varphi_{k-1} = \frac{1}{k-1} \left( 1 - k\varphi_k - \sum_{j=2}^{k-1}(k-j)\varphi_{k-j} \right).$
Then
\begin{align}
I_1 &=\lambda_1- \frac{\alpha}{1-\alpha} \varphi_{k-1}
\geq (k-1)\gamma - \frac{\alpha}{1-\alpha} \frac{1}{k-1} \left(1 - k\varphi_k - \sum_{j\ge2}(k-j)\varphi_{k-j}\right)\nonumber\\
&{\overset{\eqref{eq:varphik}}=} (k-1)\gamma - \frac{\alpha}{1-\alpha} \frac{1}{k-1}  \left(1 - k\bigl((k-1)\gamma + s + \sum_{j\ge2} (j-1)\varphi_{k-j}\bigr) - \sum_{j\ge2} (k-j)\varphi_{k-j}\right)\nonumber \\
&=(k-1)\gamma - \frac{\alpha}{1-\alpha} \frac{1}{k-1}-\frac{\alpha}{1-\alpha} \left( - k\gamma - \frac{k}{k-1} s - \sum_{j\ge2} j\varphi_{k-j} \right) \nonumber\\
& =\frac{k}{k-1} \frac{\alpha}{1-\alpha} s + \frac{\alpha}{1-\alpha} \sum_{j\ge2} j\varphi_{k-j},
\label{eq:I1}
\end{align}
where the last equality holds as $\gamma = \frac{\alpha}{(k-1)(k-1+\alpha)}$.
Similarly,
\begin{align*}
I_2 &=\lambda_2 - \frac{\alpha}{1-\alpha} \varphi_{k-1}-\frac{1+\alpha}{1-\alpha} \varphi_{k-2}\\
&\ge \frac{s}{k-1} + \frac{k}{k-1} \frac{\alpha}{1-\alpha} s + \sum_{j\ge3} \varphi_{k-j} + \frac{\alpha}{1-\alpha} \sum_{j\ge3} j\varphi_{k-j}.
\end{align*}

In general,
\begin{align}\label{eq:Ii2}
    I_i \ge \left( \frac{i-1}{k-1} + \frac{k\alpha}{(k-1)(1-\alpha)} \right) s + \sum_{j>i} \left( i-1 + \frac{\alpha}{1-\alpha} j \right) \varphi_{k-j}\ge \frac{k\alpha}{(1-\alpha)(k-1)} s  > \mu_1,
\end{align}
where $\mu_1 = \frac{k\alpha}{(1-\alpha)(k-1)} \mu$ and the last inequality holds as $s\ge \mu$.

For all $l\in[1,k-1]$, we have $\frac{l-1+\alpha'}{1-\alpha'}-\frac{l-1+\alpha}{1-\alpha}=\frac{l(\alpha'-\alpha)}{(1-\alpha')(1-\alpha)}$.
Hence, using~\eqref{eq:alpha'} and the fact that $\sum_{l=1}^{i}\varphi_{k-l}\leq 1$, we have
\begin{align}\label{eq:mu1}
  \sum_{l=1}^i \varphi_{k-l} \left( \frac{l-1+\alpha'}{1-\alpha'} - \frac{l-1+\alpha}{1-\alpha} \right)\le \mu_1< I_i.
\end{align}
So \[\lambda_i - \sum_{l=1}^i \frac{l-1+\alpha'}{1-\alpha'} \varphi_{k-l}=I_i - \sum_{l=1}^i \varphi_{k-l} \left( \frac{l-1+\alpha'}{1-\alpha'} - \frac{l-1+\alpha}{1-\alpha} \right) > 0.\]
Thus~\eqref{eq:Ii} holds, and Step~\ref{step2} can be carried out.
\end{proof}

\subsection{Handling exceptional vertices}\label{subsec:exvtx}

Let $h=|V(H)|$. The general subcase $s\geq \mu$ follows from the following lemma.

\begin{lemma}\label{lem:generalcase}
Assume $\varepsilon \le \theta \ll \rho \ll 1$, and let $G$ be an $n$-vertex graph satisfying the Ore-type condition~\eqref{eq:degree}, containing an exceptional vertex set $V_0$ with $|V_0| \le \theta n$. If the
subgraph $G'' = G \setminus V_0$ has an $\mathcal{R}_{\varepsilon,k}^{\alpha'}(L_1)$-factor with $\alpha' = \alpha +\rho$ and $L_1$ is sufficiently large, then $G$ contains an $H$-tiling that covers all but at most $\left(\dfrac{5k^2}{(k-1)^2\gamma} \omega+h\right)$ vertices.
\end{lemma}

Before proving Lemma~\ref{lem:generalcase}, we introduce some notation. Let $\mathbf{R}$ be the $\mathcal{R}_{\varepsilon,k}^{\alpha'}(L_1)$-factor of $G''$, and let $\mathcal{R}\in \mathbf{R}$ be an unbalanced $\varepsilon$-regular $k$-clique $\mathcal{R}_{\varepsilon,k}^{\alpha'}(L_1)$. We write $\mathcal{R}=\{U_1,\ldots,U_k\}$, where $U_k$ denotes the smallest cluster, of size $\alpha' L_1$, and the other clusters have size $L_1$. When no confusion arises, we simply call $\mathcal{R}$ a \emph{clique}.

\begin{proof}[Proof of Lemma~\ref{lem:generalcase}]
The proof consists of two phases. In Phase I, we insert the exceptional vertices into suitable regular $k$-cliques and cover them by copies of $H$. In Phase II, we use the connections between different regular $k$-cliques to reduce the number of uncovered vertices to a constant depending only on $H$.

First, partition $V_0$ into two subsets. Let
$V_0^1 = \left\{x \in V_0 : d_G(x) < \left(1 - \frac{1}{k-1} + \gamma\right)n\right\}$ and $V_0^2 = V_0 \setminus V_0^1$. By the Ore-type condition~\eqref{eq:degree}, $V_0^1$ induces a clique in $G$. Then $G[V_0^1]$ can be tiled by copies of $H$, leaving at most $h-1$ vertices uncovered. It remains to cover the vertices of $V_0^2$. We will remove pairwise disjoint copies of $H$, each containing exactly one vertex of $V_0^2$, while ensuring that no cluster in the factor $\mathbf{R}$ loses too many vertices.

\medskip
\noindent\textbf{Phase I: Covering the exceptional vertices.}

We shall use the following corollary of the Key Lemma~\cite{komlos1995}.

\begin{lemma}[Koml\'{o}s and Simonovits~\cite{komlos1995}]\label{lem:komlos1995}
    Let $\mathcal{R} = \{V_1, \ldots, V_k\}$ be an $\varepsilon$-regular $k$-clique in $G''$. Assume $W_i \subset V_i$ with $|W_i| = d' L$ for all  $i \in[1, k]$, where $d' \gg \varepsilon$. Let $\mathcal{R}_0 = \{W_1, \ldots, W_k\}$. Then $K_k(h) \subset \mathcal{R}_0$. In particular, $H \subset \mathcal{R}_0$, and we can put its small color-class in any $W_i$.
\end{lemma}

For a vertex $v \in V_0^2$ and a cluster $U \in V(R)$, we write $v \sim U$ if $d_G(v, U) \ge d|U|$.

The procedure of Phase I is as follows. We process the vertices of $V_0^2$ in any order. For
each $v \in V_0^2$, we pick an element $\mathcal{R} = \{U_1, U_2, \ldots, U_k\}\in \mathbf{R}$ such that
\begin{align}\label{eq:vUi}
    v \sim U_i \ \text{for at least}\ k-1\ \text{clusters}\ U_i \in \{U_1, \ldots, U_k\}.
\end{align}

Then we construct a copy of $H$ containing $v$ and vertices from $\mathcal{R}$ as follows.
\begin{itemize}
    \item If $v \sim U_i$ for all $i\in[1, k-1]$, then we place $v$ in the small color-class of $H$. Thus we remove a copy of $H$ consisting of $v$, $\sigma-1$ vertices from $U_k$, and $\omega$ vertices from each $U_i$ for $i \in[1, k-1]$.
    \item Otherwise, there exists some $j \neq k$ such that $v \sim U_i$ for all $i \neq j$. Then we place $v$ in the color-class corresponding to $U_j$. Thus we remove $\omega-1$ vertices from $U_j$, $\sigma$ vertices from $U_k$, and $\omega$ vertices from each $U_i$ with $i\notin \{j,k\}$.
\end{itemize}

In both cases, Lemma~\ref{lem:komlos1995} guarantees the required copy of $H$, since $v$ has many neighbors in each cluster used for the other color-classes. To prevent a cluster from losing too many vertices, we stop considering an element of $\mathbf{R}$ if it has been selected $\theta_1 L_1$ times, where $\theta_1=\sqrt{\theta}$. The following proposition guarantees that an available element satisfying~\eqref{eq:vUi} always exists.

\begin{prop}\label{prop:propotion}
 For any vertex $v \in V_0^2$, regardless of the processing order, at least $\alpha/2$ proportion of the elements in $\mathbf{R}$ satisfy~\eqref{eq:vUi}.
\end{prop}

\begin{proof}
Let $v\in V_0^2$ be the $(t+1)$-st vertex to be processed, where $t$ denotes the number of vertices of $V_0^2$ that have already been processed. For these first $t$ vertices, we have removed $t$ pairwise disjoint copies of $H$, each containing exactly one processed vertex of $V_0^2$ and $h-1$ vertices from elements of $\mathbf{R}$.
Let $G^*$ be the subgraph induced by the vertices that remain after these $t$ copies of $H$ have been removed. Since $v\in V_0^2$, we have $d_G(v)\geq \left(1-\frac{1}{k-1}+\gamma\right)n$ and
\begin{align}\label{eq:propotion-lower}
d_G(v,V(G^*))\geq\left(1-\frac{1}{k-1}+\gamma\right)n-\theta n-t(h-1).
\end{align}

Let $m$ be the fraction of elements in $\mathbf{R}$ satisfying~\eqref{eq:vUi}. We derive a lower bound on $m$. Recall that every element $\mathcal{R}\in \mathbf{R}$ has size $(k-1+\alpha')L_1$. If $\mathcal{R}$ does not satisfy~\eqref{eq:vUi}, then at least two clusters $U$ fail $v\sim U$, and hence each such cluster contributes fewer than $d|U|$ neighbors of $v$. Thus such a bad element contributes at most $(k-2)L_1+dL_1+d\alpha'L_1
=(k-2+d+\alpha'd)L_1$ neighbors of $v$. We have the upper bound on  $d_G(v, V(G^*))$.
\begin{align}\label{eq:propotion-upper}
d_G(v, V(G^*))\le \left( (1 - m) \frac{k-2 + \alpha' d + d}{(k-1) + \alpha'} + m \right) \bigl(n - \theta n - (h-1)t\bigr).
\end{align}
Combining~\eqref{eq:propotion-lower} and~\eqref{eq:propotion-upper}, and using $\theta n+t(h-1)\leq \theta hn$, we obtain
\begin{align*}
1 - \frac{1}{k-1} + \gamma - \theta h \le 1 - \frac{1 + \alpha' - \alpha' d - d}{(k-1) + \alpha'} + \frac{1 + \alpha' - \alpha' d - d}{(k-1) + \alpha'} m,
\end{align*}
which implies
\[\frac{1 + \alpha' - \alpha' d - d}{(k-1) + \alpha'} - \frac{1}{k-1 + \alpha} - \theta h \le \frac{1 + \alpha' - \alpha' d - d}{(k-1) + \alpha'} m.
\]
Therefore,
\[ m > \frac{\alpha-2d-kh\theta}{1+\alpha'-d-\alpha'd}> \frac{\alpha}{2}\]
holds since we may assume that $2d+kh\theta<\frac{\alpha}{2}$.
\end{proof}

Since no element of $\mathbf{R}$ is selected more than $\theta_1L_1$ times, the proportion of unavailable elements is at most
\[ \frac{\dfrac{\theta n}{\sqrt{\theta} L_1}}{\dfrac{n}{(k-1+\alpha')L_1}} < \sqrt{\theta} k. \]
By choosing $\theta$ sufficiently small, we have $k\sqrt{\theta}<\frac{\alpha}{2}$.
Hence Proposition~\ref{prop:propotion} ensures that Phase I can be carried out.

After all vertices of $V_0^2$ have been processed, and after tiling $G[V_0^1]$, at most $h-1$ original exceptional vertices remain uncovered. We now make the regular pairs super-regular. Let $\mathcal{R}=\{U_1, \ldots, U_k\}\in\mathbf{R}$. For any vertex $v \in U_i$, if there exists $j \neq i$ with $d_G(v, U_j) < (d - \varepsilon)|U_j|$, then we move $v$ to $V_0$. By $\varepsilon$-regularity, each cluster loses at most $(k-1)\varepsilon |U_i|$ vertices in this way. Since $\varepsilon\ll\theta$, these newly exceptional vertices can be handled by the same procedure as above, after again separating them into low-degree and high-degree vertices. This does not affect the preceding estimates. After this cleaning step, every pair inside each remaining element of $\mathbf{R}$ is $(\varepsilon,d/2)$-super-regular.

\medskip
\noindent\textbf{Phase II: Reducing the number of uncovered vertices.}

Let $\mathbf{R}'$ denote the resulting clique-cover at the end of Phase I. Since some vertices have been removed from the clusters during Phase I, the sizes of resulting clusters vary.
Recall that $\mathcal{R}^{\alpha'}_k(L_1) \supseteq \mathcal{R}^\alpha_k(L_2) \cup \mathcal{R}_k^1(L_0)$, where $L_2 = \frac{1-\alpha'}{1-\alpha} L_1$ and $L_0 = \frac{\alpha'-\alpha}{1-\alpha} L_1$. Since $\alpha'-\alpha=\rho \gg \theta_1$, the number of vertices removed from each cluster in Phase I is much smaller than $L_0$. Hence each element $\mathcal{R}' \in \mathbf{R}'$ contains at least one balanced $\varepsilon$-regular $k$-clique $\mathcal{R}_k^1(L_0/2)$. Thus, $\mathcal{R}'$ contains an $H$-tiling, where the $\sigma$-vertex classes can be taken from any cluster of element $\mathcal{R}'$.

This flexibility allows us to transfer unused vertices among the clusters of a fixed element $\mathcal{R}'$. Indeed, suppose we want to move $\omega-\sigma$ unused vertices from a cluster $U_1$ to another cluster $U_2$ in $\mathcal{R}'$. We simply switch the color-classes for one
copy of $H$ such that the $\sigma$-vertex class which was supposed to come from $U_1$ will come
from $U_2$. This results in a loss of $\omega - \sigma$ vertices from $U_1$ while  $U_2$ will gain $\omega - \sigma$ vertices. This observation will help us establish the following result.

\begin{prop}\label{prop:R'}
Each element of $\mathbf{R}'$ has an $H$-tiling that covers all but at most $(k-1)(2\omega - \sigma) +\omega$ vertices.
\end{prop}

\begin{proof}
Consider an element $\mathcal{R}'=\{U_1,\ldots,U_k\}\in \mathbf{R}'$. Since all cluster-pairs in $\mathcal{R}'$ are super-regular, we can apply the Blow-up Lemma (Lemma~\ref{lemma:blowup}) to $\mathcal{R}'$. We pick an $H$-tiling of $\mathcal{R}'$ (which leaves $a_1, a_2, \ldots, a_k$ vertices uncovered from clusters $U_1, U_2, \ldots, U_k$, respectively), satisfying the following two conditions:
\begin{itemize}
    \item It is one of the best $H$-tilings, i.e., $\sum_{i=1}^{k} a_i$ is the minimum over all $H$-tilings of $\mathcal{R}'$.
    \item It is the most balanced, i.e., $\sum_{i=1}^{k} |a_i -\omega|$ is the smallest among the best tilings.
\end{itemize}

Assume that $a_{i_0} = \min\{a_i:i\in[1,k]\}$. We claim that $a_{i_0} < \omega$. Indeed, if $a_i\geq \omega$ for all $i\in[1,k]$, then, since $\sigma\leq \omega$, the uncovered vertices in the $k$ clusters would contain enough vertices to form one additional copy of $H$. This contradicts the maximality of the chosen tiling. Moreover, we claim that for all $i \neq i_0$, we have $a_i < 2\omega- \sigma$. Suppose otherwise that  $a_i\ge 2\omega- \sigma$  for some  $i\neq i_0$, then we can move $(\omega- \sigma)$ unused vertices from $U_i$ to $U_{i_0}$. This preserves $\sum_{i=1}^{k} a_i$, but decreases $\sum_{i=1}^{k} |a_i - \omega|$. A contradiction to the choice of the tiling. Therefore,
\begin{equation*}
\sum_{i=1}^{k} a_i = a_{i_0} + \sum_{i \neq i_0} a_i < \omega + (k - 1)(2\omega - \sigma).
\end{equation*}
Thus $\mathbf{R}'$ has an $H$-tiling that covers all but at most $(k-1)(2\omega - \sigma) +\omega$ vertices.
\end{proof}

If we apply Proposition~\ref{prop:R'} directly to all elements of $\mathbf{R}'$, we obtain an $H$-tiling of $G''$ that leaves at most $\bigl((k-1)(2\omega - \sigma) + \omega\bigr)|\mathbf{R}'|$ vertices uncovered. Recall that $|\mathbf{R}'| = O(\ell)$ and $\ell\le M(\varepsilon)$ (from Lemma~\ref{lemma:Regularity}). However, we need a bound depending only on $H$, and hence independent of $\varepsilon$. To achieve this, we use the connections between different elements of $\mathbf{R}'$.

Define a directed graph $\mathcal{D}$ as follows. The vertex set $V(\mathcal{D})$ consists of the clusters of $G''$. A directed edge $(U,U_1)\in E(\mathcal{D})$ if and only if $U_1\in\mathcal{R}' = \{U_1, U_2, \ldots, U_k\}\in \mathbf{R}$ and $(U, U_i) \in E(R)$ for all $i \neq 1$. Thus, a directed edge $(U,U_1)$ means that one vertex from $U$ can be used together with vertices from the other $k-1$ clusters of $\mathcal{R}$ to form a copy of $H$. In this way, an extra vertex in $U$ can be moved to $U_1$.

Since Phase I removes only a small number of vertices from each cluster, the adjacency in the reduced graph $R$, and hence the directed graph $\mathcal{D}$, is unchanged.

In a directed graph $D$, the \emph{source set} of a vertex $v$ is defined as $\mathcal{W}(v) = \{u \in V(D) : \text{there exists a directed path from } u \text{ to } v\}$. It is easy to see that $|\mathcal{W}(u)| \ge |\mathcal{W}(v)|$ if $u \in N^+(v)$. A set $\mathcal{S}(D) \subseteq V(D)$ is called a \emph{sink set} if $V(D) = \bigcup_{v \in \mathcal{S}(D)} \mathcal{W}(v)$. We will show that $\mathcal{D}$ has a sink set $\mathcal{S}(\mathcal{D})$ of size at most $\dfrac{5k^2}{(k-1)^2\gamma} \omega$.

We now describe the transfer procedure.
\begin{enumerate}[{\bf Step 1}]
\item For each element $\mathcal{R}'\in\mathbf{R}'$, choose a largest $H$-tiling as in Proposition~\ref{prop:R'}, but do not yet remove the corresponding copies. For each cluster $C$, let $\operatorname{extra}(C)$ denote the number of vertices of $C$ left uncovered by this tiling.
\item Let $\mathcal{S}(\mathcal{D})$ be a sink set. For every cluster $C\notin \mathcal{S}(\mathcal{D})$ with
    $\operatorname{extra}(C)>0$, choose a directed path $C=C^0,C^1,\ldots,C^t$ in $\mathcal{D}$, where $C^t\in\mathcal{S}(\mathcal{D})$.

    We transfer the extra vertices of $C$ along this path. For each $i\in[1,t]$, let $\mathcal{R}_i\in\mathbf{R}'$ be the element containing $C^i$. Since $(C^{i-1},C^i)\in E(\mathcal{D})$, the cluster $C^{i-1}$ is adjacent in $R$ to every cluster of $\mathcal{R}_i$ except possibly $C^i$. Therefore, using the Blow-up Lemma, we can remove copies of $H$, each containing one vertex from $C^{i-1}$ and the remaining $h-1$ vertices from the other clusters of $\mathcal{R}_i$.

    This operation decreases $\operatorname{extra}(C^{i-1})$ and increases $\operatorname{extra}(C^i)$ by the same amount. Repeating this along the directed path moves all extra vertices originally in $C$ into the sink set.

    Since the total number of extra vertices being transferred is bounded by a constant and $L_1$ is sufficiently large, these additional removals do not affect the super-regularity of the relevant pairs.

    \item After all extra vertices have been transferred into the sink set, apply the Blow-up Lemma to the remaining vertices in each element of $\mathbf{R}'$. Combining these copies of $H$ with the copies removed in Phase I and during the transfer process gives an $H$-tiling of $G$.
\end{enumerate}

It remains to prove that $\mathcal{D}$ has a small sink set. Let \[J_L=\left\{U\in V(\mathcal{D}):d_R(U)<\left(1-\frac{1}{k-1}+\gamma-d-2\varepsilon\right)|V(R)|\right\}\] and let $J_H=V(\mathcal{D})\setminus J_L$. By the Ore-type condition~\eqref{eq:oreconditionR}, the set $J_L$ induces a clique in $R$. Hence the extra vertices transferred to clusters in $J_L$ can be handled in this clique, leaving at most $k$ additional vertices uncovered. Thus it remains to find a small sink set for $J_H$.
We use the following elementary lemma.

\begin{lemma}\label{lem:sinkset}
Let $D$ be a directed graph with minimum out-degree at least $\delta$. Then $D$ has a sink set of size at most
$\frac{|V(D)|}{\delta+1}$.
\end{lemma}

\begin{proof}Let $x_1\in V(D)$ such that  $|\mathcal{W}(x_1)| = \max_{v \in V(D)} |\mathcal{W}(v)|$. We claim that
$N_D^+(x_1)\subseteq \mathcal{W}(x_1)$. For any $u\in N_D^+(x_1)$, any directed path to $x_1$ can be extended to $u$ via the edge $(x_1, u)$, which implies $\mathcal{W}(x_1) \subseteq \mathcal{W}(u)$.
By the maximality of $|\mathcal{W}(x_1)|$, we have $\mathcal{W}(x_1) = \mathcal{W}(u)$, and therefore $u\in\mathcal{W}(x_1)$. It follows that $|\mathcal{W}(x_1)|\geq d_D^+(x_1)+1\geq \delta+1$.

Delete $\mathcal{W}(x_1)$ and repeat the same argument in the remaining directed graph. If a vertex outside $\mathcal{W}(x_1)$  had an out-neighbour in $\mathcal{W}(x_1)$, then it could reach $x_1$, a contradiction. Thus the minimum out-degree in the remaining digraph is still at least $\delta$. Hence at least $\delta+1$ vertices are deleted at each step. The selected vertices form a sink set $S(D)$, and 
$|S(D)|\le\frac{|V(D)|}{\delta+1}$.
\end{proof}

\begin{prop}\label{prop:degreeU}
For any $U \in V(\mathcal{D})$ with $d_{R}(U)\ge(1-\frac{1}{k-1}+\gamma-d-2\varepsilon)|V(R)|$, we have
 $d_{\mathcal{D}}^{+}(U) \geq \frac{(k-1)^2 \gamma}{k} |V(\mathcal{D})|$.
\end{prop}
\begin{proof}
Since the connections in $R$ are not influenced by the insertion of exceptional vertices in Phase I, we may assume that the sizes of clusters in each clique $\mathcal{R}\in \mathbf{R}$ are still $L_1$ and $\alpha' L_1$.

For a given $U \in V(\mathcal{D})$, let $m_1$, $m_2$, and $m_3$ denote the fraction of cliques $\mathcal{R}=\{U_1, U_2, \ldots, U_k\} \in \mathbf{R}$ for which $(U, U_i) \in E(R)$ for all $i\in [1,k-1]$ but $(U, U_k) \notin E(R)$, $(U, U_i) \in E(R)$ but $(U, U_j) \notin E(R)$ for all $i \neq j$ with $j \neq k$, and $(U, U_i) \in E(R)$ for all $i\in [1,k]$, respectively. By the degree condition of $U$ in the assumption, we have
\begin{align*}
1 - \frac{1}{k-1} + \gamma - d -2\varepsilon &\leq \frac{k-2}{k-1+\alpha'} + \frac{1}{k-1+\alpha'} m_1  + \frac{\alpha'}{k-1+\alpha'} m_2 + \frac{1+\alpha'}{k-1+\alpha'} m_3 \nonumber\\
&< \frac{k-2}{k-1+\alpha'} + \frac{1}{k-1+\alpha'} m_1 + \frac{1}{k-1+\alpha'} m_2 + \frac{2}{k-1+\alpha'} m_3 \nonumber \\
 &< \frac{k-2}{k-1+\alpha} + \frac{1}{k-1+\alpha} (m_1 + m_2 + 2m_3)\quad (\text{as}\ \alpha<\alpha'<1)\nonumber\\
&= \left(1 - \frac{1}{k-1} - (k-2)\gamma\right) + \left(\frac{1}{k-1} - \gamma\right)(m_1 + m_2 + 2m_3),
\end{align*}
that is,
\[
(k-1)\gamma\leq \left(\frac{1}{k-1} - \gamma\right)(m_1 + m_2 + 2m_3).
\]
By the definition of $\mathcal{D}$, $U$ has exactly one out-neighbor into $\mathcal{R}$ for the cases corresponding to $m_1$ and $m_2$, and $k$ out-neighbors for the case corresponding to $m_3$. Since $k \ge 2$, we have $m_1 + m_2 + k m_3 \ge m_1 + m_2 + 2m_3$. Note that $|\mathbf{R}| = \frac{|V(\mathcal{D})|}{k}$, we have
\begin{equation*}
d_{\mathcal{D}}^{+}(U) = (m_1 + m_2 + k m_3)|\mathbf{R}|
\geq \frac{(k-1)\gamma}{\frac{1}{k-1} - \gamma} \frac{|V(\mathcal{D})|}{k}
> \frac{(k-1)^2 \gamma}{k} |V(\mathcal{D})|.\qedhere
\end{equation*}
\end{proof}

Let $\mathcal{W}(J_L)=\bigcup_{U\in J_L}\mathcal{W}(U)$. Every cluster in $\mathcal{W}(J_L)$ can reach a low-degree cluster, and hence its extra vertices can be transferred into $J_L$. As noted above, the extra vertices in $J_L$ can be handled with at most $k$ additional uncovered vertices.
Now consider the induced subdigraph $\mathcal{D}_H=\mathcal{D}-\mathcal{W}(J_L)$. 
If $\mathcal{D}_H$ is non-empty, then every vertex of $\mathcal{D}_H$ belongs to $J_H$. Moreover, no vertex of $\mathcal{D}_H$ has an out-neighbour outside $\mathcal{D}_H$; otherwise it could reach $J_L$. Therefore Proposition~\ref{prop:degreeU} gives $\delta^+(\mathcal{D}_H)\geq\frac{(k-1)^2\gamma}{k}|V(\mathcal{D})|$.
By Lemma~\ref{lem:sinkset}, $\mathcal{D}_H$ has a sink set $S(\mathcal{D})$ with \[|S(\mathcal{D})|\le\frac{|V(\mathcal{D})|}{\frac{(k-1)^2\gamma}{k}|V(\mathcal{D})|+1}\leq\frac{k}{(k-1)^2\gamma}.\] 
Thus, after all extra vertices have been transferred, the only uncovered vertices lie in at most
$\frac{k}{(k-1)^2\gamma}$ high-degree sink clusters, together with at most $k$ additional vertices arising from the low-degree clique structure. By Proposition~\ref{prop:R'}, the total number of uncovered vertices is therefore at most \[((k-1)(2\omega-\sigma)+\omega)\left(\frac{k}{(k-1)^2\gamma}\right)+k+h\le \frac{5k^2}{(k-1)^2\gamma}\omega+h\]
since $\sigma\leq\omega$.
This completes the proof of Lemma~\ref{lem:generalcase}.
\end{proof}

\section{The extremal subcase when $s<\mu$}\label{sec:excase}

By Lemma~\ref{lem:generalcase}, the general subcase where $s \ge \mu$ was resolved in Section~\ref{sec:general}. In this section, we assume that $s < \mu$. Recall that $\Phi = \{\Phi_k, \ldots, \Phi_1\}$ is a maximal clique-cover of $R$. Substituting~\eqref{eq:varphik} into~\eqref{eq:sum1}, we have
\begin{align*}
(k-1)\varphi_k + (k-1)\gamma + \sum_{i=2}^{k-1} \varphi_{k-i}(i-1) + s + (k-1)\varphi_{k-1} + \sum_{i=2}^{k-1} \varphi_{k-i}(k-i) = 1,
\end{align*}
which implies
\begin{align}\label{eq:sum}
\sum_{i=1}^{k}\varphi_i= \frac{1}{k-1} - \gamma - \frac{s}{k-1} \geq \frac{1}{k-1} - \gamma - \frac{\mu}{k-1}.
\end{align}
Also from~\eqref{eq:sum1}, we get
\begin{align*}
(k-1)k\gamma + ks + (k-1)\sum_{j \geq 1} j\varphi_{k-j} = 1.
\end{align*}
For $i < k$, define $\sigma_i = \sum_{j \leq i} (k-j)\varphi_j$. We have
\[
\sigma_{k-1} = \frac{1}{k-1} - k\gamma - \frac{k}{k-1}s = \frac{1-\alpha}{k-1+\alpha} - \frac{k}{k-1}s \geq \alpha_0,
\]
where $\alpha_0 = \frac{1-\alpha}{k-1+\alpha} - \frac{k}{k-1}\mu$.

Let $C = \sqrt[k]{\alpha_0/\varepsilon}$. Since $\sigma_{k-1} \geq \alpha_0 = C^k \varepsilon$, there exists an integer $t\in[1, k-1]$ and an integer $j_0 \leq k$ such that $\sigma_t \geq C^{j_0}\varepsilon$ and $\sigma_{t-1} \leq C^{j_0-1}\varepsilon$. Set $\varepsilon' = k(C^{j_0-1}\varepsilon + 2\mu)$ and $\mu' = C^{j_0}\varepsilon$.

We first move all vertices of $G$ in $\Phi_j$ for $j < t$ into $V_0$. The size of the resulting exceptional set (still denoted by $|V_0|$) is bounded by
\[
|V_0|<\varepsilon n + \sum_{j < t} j\varphi_j n < \varepsilon n + k\sum_{j < t} \varphi_j n < \varepsilon n + k\sigma_{t-1}n \leq  \varepsilon n +kC^{j_0-1}\varepsilon n < \varepsilon' n.
\]

Following the proof of Proposition~\ref{prop:varphik}, it is not hard to show that the number of $k$-cliques that are over-connected to some smaller clique is at most $2t\mu\ell$. Since the size of $\Phi_t$ is not small, most $t$-cliques are not over-connected to any $k$-clique. Fix two such $t$-cliques $K_1,K_2\in \Phi_t$. For $j\in [0, k-t]$, let $\Phi_{k-j}^0 = \{K \in \Phi_{k-j} : e(\{K_1, K_2\}, K) < 2t(k-j-1)\}$ and $m_j = |\Phi_{k-j}^0|/\ell$. Using similar computations as in Proposition~\ref{prop:lambdai}, we obtain
\[
m_0 + m_1 + \ldots + m_{k-t} < \frac{2t}{k-1}s + 4dt < \frac{\varepsilon'}{k}.
\]

Recall that in Proposition~\ref{prop:connectivity}(iii), we divide each clique $K^j$ ($j \geq t$) which is well-connected to $K_1, K_2$ into sets $A_t(K^j)$ and $B_t(K^j)$, where $A_t(K^j)$ consists of the clusters which are adjacent to all but one of $K_1$ (or $K_2$), and $B_t(K^j)$ are those which are adjacent to all of $K_1$ (or $K_2$). When combining $A_t(K)$ for all clusters $K \notin \Phi_t^0 \cup \ldots \cup \Phi_k^0$, we obtain a cluster set $A$ such that
\begin{align*}|A|=t\left(\sum_{j=t}^{k}\varphi_j-\sum_{j=0}^{k-t}m_j\right)\ell
{\overset{\eqref{eq:sum}}\geq} t\left(\frac{1}{k-1} - \gamma - \frac{2k\mu}{k-1}\right)\ell.
\end{align*}
On one hand, $A$ is covered by a family of $t$-cliques. On the other hand, by Proposition~\ref{prop:connectivity}(iii), $A$ is the disjoint union of $t \leq k-1$ independent sets $U_1, \ldots, U_t$, with $|U_j| \geq\left(\frac{1}{k-1} - \gamma - \frac{2k\mu}{k-1}\right)\ell\ge \left(\frac{1}{k-1} - \gamma - \frac{\varepsilon'}{k-1}\right)\ell$.
Because each $U_i$ forms an independent set in $R$, any two vertices in $U_i$ are nonadjacent. By~\eqref{eq:oreconditionR}, the sum of their degrees must be at least $2\left(1-\frac{1}{k-1}+\gamma-d-2\varepsilon\right)\ell$. So at most one vertex in each $U_i$ can have degree less than $\left(1-\frac{1}{k-1}+\gamma-d-2\varepsilon\right)\ell$.
Thus, after moving at most one vertex from each $U_i$ into $V_0$, we may assume that every $c\in U_i$ satisfies
$d_R(c)\geq\left(1-\frac{1}{k-1}+\gamma-d-2\varepsilon\right)\ell$.
Since $U_i$ is independent, all neighbors of $c\in U_i$ lie in $R\setminus U_i$. Moreover,
$|R\setminus U_i|\leq\ell-\left(\frac{1}{k-1}-\gamma-\frac{\varepsilon'}{k-1}\right)\ell$.
Hence, for all $c\in U_i$, we have
\begin{align}\label{eq:degreec}
d_R(c,R\setminus U_i)=d_R(c)\geq\left(1-\frac{1}{k-1}+\gamma-d-2\varepsilon\right)\ell\geq|R\setminus U_i|-\frac{\varepsilon'}{k-1}\ell.
\end{align}

In other words, all but at most one element of $U_i$ are almost adjacent to all the clusters outside $U_i$.
Depending on the value of $t$, we will consider two separate cases.

\subsection{Extremal subcase (I): $t=k-1$}

Let $U_k = V(R) \setminus \bigcup_{i=1}^{k-1}U_i$, and we also write $U_i$ (for $1 \le i \le k$) for the underlying vertex sets in $G$. We first move all exceptional vertices from $V_0$ to $U_k$. Then, we move a vertex $v \in U_k$ to $U_i$ (for some $i < k$) if $d(v, U_i) < \varepsilon_1 |U_i|$, where $\varepsilon' \ll \varepsilon_1 \ll 1$. The resulting sets are still denoted by $U_1, \ldots, U_k$.

In the ideal case, $|U_i| = \left(\frac{1}{k-1} - \gamma\right)n$ for all $i < k$ and $|U_k| = (k-1)\gamma n$. Due to the superregularity between every pair in $\{U_1, \ldots, U_k\}$, we can apply Lemma~\ref{lemma:blowup} to find the desired $H$-factor in $G$.

If $|U_i| < \left(\frac{1}{k-1} - \gamma\right)n$ for all $i < k$, we can construct copies of $H$ by placing their $\omega$-classes in $U_k$. An argument analogous to Proposition~\ref{prop:R'} then shows that all but at most $2k\omega$ vertices of $G$ are covered by an $H$-tiling. Thus, we may assume $|U_1| > \left(\frac{1}{k-1} - \gamma\right)n$, while $|U_i| < \left(\frac{1}{k-1} - \gamma\right)n$ for $1 < i < k$, and $|U_k| < (k-1)\gamma n$. The other cases are symmetric.

To balance the parts, we move vertices from $U_1$ to the deficient sets. More precisely, whenever there exists a vertex $v\in U_1$ satisfying $d_G(v,U_1)>\varepsilon_1|U_1|$, we move $v$ into some set $U_i$, $i>1$, whose size is below the ideal value. This operation preserves the super-regularity of the pairs between distinct parts, since $v$ has many neighbors outside $U_1$. We continue this process until either the ideal sizes are reached, or every vertex $v\in U_1$ satisfies $d_G(v,U_1)\leq \varepsilon_1|U_1|$.
In the first case, the Blow-up Lemma completes the proof. Thus we may assume the latter condition holds.
Let $x=|U_1|-\left(\frac{1}{k-1}-\gamma\right)n$.
We partition $U_1$ into $U_H=\left\{v\in U_1:d_G(v)\geq\left(1-\frac{1}{k-1}+\gamma\right)n\right\}$ and $U_L=U_1\setminus U_H$.
For any $v\in U_H$, we have
\[d_G(v,U_1)=d_G(v)-d_G(v,V(G)\setminus U_1)\geq\left(1-\frac{1}{k-1}+\gamma\right)n
-\bigl(n-|U_1|\bigr)=x.\]
Since we are assuming $d_G(v,U_1)\leq \varepsilon_1|U_1|$ for all $v\in U_1$, it follows that $x\leq \varepsilon_1|U_1|$.
By the Ore-type condition~\eqref{eq:degree}, $U_L$ induces a clique in $G$. Since $\Delta(G[U_1])\leq \varepsilon_1|U_1|$, we have $|U_L|\leq \Delta(G[U_1])+1\leq \varepsilon_1|U_1|+1$.
Thus, for sufficiently large $n$, $|U_H|\geq (1-2\varepsilon_1)|U_1|$.
We will use the following proposition.

\begin{prop}\label{prop:nui}Let $i$ be a positive integer and let $\delta_{U_1}(U_H)=\min_{v\in U_H} d_G(v,U_1)$.
Then \[\nu_i(G[U_1])\geq\frac{(\delta_{U_1}(U_H)-i+1)(1-2\varepsilon_1)|U_1|}{2(i+1)\Delta(G[U_1])}.\]
\end{prop}

\begin{proof}Let $\mathcal{S}$ be a maximum collection of $i$-stars in $G[U_1]$ and let $t=|\mathcal{S}|$. Note that $|V(\mathcal{S})|=t(i+1)$. Denote by $U=U_1\setminus V(\mathcal{S})$. We will bound $e_G(U, V(\mathcal{S}))$ as follows. By the maximality of $\mathcal{S}$, $G[U]$ contains no $i$-star and so $\Delta(G[U])\le i-1$. For any vertex $v\in U\cap U_H$, we have
\[d(v, V(\mathcal{S}))= d(v, U_1) - d(v, U) \ge \delta_{U_1}(U_H)- (i-1). \]
It implies that
\begin{align*}e_G(U, V(\mathcal{S}))&\ge e_G(U\cap U_H, V(\mathcal{S})) \ge |U \cap U_H| (\delta_{U_1}(U_H) - i + 1)\\
&\ge (|U_H| - |V(\mathcal{S})|)(\delta_{U_1}(U_H)- i + 1) \ge ((1-2\varepsilon_1)|U_1| - t(i+1))(\delta_{U_1}(U_H)- i + 1).
\end{align*}

On the other hand,
\begin{align*}
   e_G(U, V(\mathcal{S})) \le \sum_{v \in V(\mathcal{S})} d(v, U_1) \le |V(\mathcal{S})| \Delta(G[U_1]) = t(i+1)\Delta(G[U_1]).
\end{align*}
Combining the lower bound and upper bound on $e_G(U, V(\mathcal{S}))$, we obtain
\begin{align*} ((1-2\varepsilon_1)|U_1| - t(i+1))(\delta_{U_1}(U_H) - i + 1) \le t(i+1)\Delta(G[U_1]). \end{align*}
It follows that
\begin{align*}
 t\ge \frac{(1-2\varepsilon_1)|U_1|(\delta_{U_1}(U_H)- i + 1)}{(i+1)(\Delta(G[U_1])+\delta_{U_1}(U_H)-i+1)}\ge \frac{(1-2\varepsilon_1)|U_1|(\delta_{U_1}(U_H)- i + 1)}{2(i+1)\Delta(G[U_1])},
 \end{align*}
where the last inequality holds as $\delta_{U_1}(U_H)- i + 1 \le \Delta(G[U_1])$.
\end{proof}

We now apply Proposition~\ref{prop:nui} with $i=\omega$. Since $\delta_{U_1}(U_H)\geq x$ and $\Delta(G[U_1])\leq \varepsilon_1|U_1|$, we have
\begin{align*}
    \nu_\omega(G[U_1]) \ge (x - \omega + 1) \frac{(1-2\varepsilon_1)|U_1|}{2(\omega+1)\varepsilon_1|U_1|}=(x - \omega + 1) \frac{(1-2\varepsilon_1)}{2(\omega+1)\varepsilon_1}
\end{align*}

If $x\leq 2\omega$, then the imbalance is bounded by a constant depending only on $H$, and we may leave these vertices uncovered. Thus we may assume that $x>2\omega$. Since $\varepsilon_1$ is chosen sufficiently small, so $\nu_\omega(G[U_1])\geq x$ and $G[U_1]$ contains $x$ vertex-disjoint $\omega$-stars. For each star, move its center from $U_1$ to one of the deficient sets $U_i,i>1$. Since the center has $\omega$ neighbors in $U_1$, we can immediately construct a copy of $H$ using this center in the appropriate color-class, its $\omega$ neighbors in $U_1$, and suitable vertices from the remaining parts. Removing these $x$ copies of $H$ decreases $|U_1|$ by exactly $x$ more vertices than in the deficient parts, thereby balancing the sets.
After this operation, the remaining sets $U_1,\ldots,U_k$ have size ratio $(1,\ldots, 1, \alpha)$
up to a bounded error depending only on $H$. Moreover, the pairs between distinct parts remain super-regular. Therefore, by applying Lemma~\ref{lemma:blowup}, to the remaining graph, we obtain an $H$-tiling covering all but a constant number of vertices. This completes the proof in the case $t=k-1$.

\subsection{Extremal subcase (II): $t<k-1$}\label{subsec:II}

Let $B = V(R)\setminus A$. We write $V_A$ for the union of the clusters of $G$ corresponding to vertices of $A$, and $V_B$ for the union of the clusters corresponding to vertices of $B$, together with the current exceptional set $V_0$. Let $H_0$ be the $(k-t)$-partite bottle graph with width $\omega$ and neck $\sigma$. Our goal is to find an almost perfect $H_0$-tiling of $V_B$ such that each copy of $H_0$ has a large common neighborhood (defined as common neighbors of the vertices
in $H_0$) in $V_A$. Thus we can apply Hall's theorem to match each copy of $H_0$ with a copy of $K_t(\omega)$ from $V_A$. This yields the desired $H$-tiling in $G$, which leaves only a constant number of vertices uncovered.

The construction of the $H_0$-tiling in $V_B$ is analogous to the procedure used in the general subcase. We first tile $V_B$ with $\mathcal{R}_{\varepsilon,k-t}^{\alpha'}(L_1)$, which are unbalanced $\varepsilon$-regular $(k-t)$-cliques. Then we move vertices of $V_0$ to $V_B$. The only difference is that we have to take special
care of the vertices in $V_B$ whose degrees in $V_A$ are small.

Let us get into some details about the tiling of $V_B$ with $\mathcal{R}_{\varepsilon,k-t}^{\alpha'}(L_1)$. We can take advantage of the existing clique-cover. By definition, $B = \{B_0, B_1, \ldots, B_{k-t}\}$, where $B_{j-t} = \bigcup_{K^j \in \Phi_j \setminus \Phi_j^0} B_t(K^j)$ for $j\in[t,k]$, and $B_0 = \Phi_k^0 \cup \ldots \cup \Phi_t^0$ denotes all $j$-cliques ($j \ge t$) that did not participate in constructing $A$. We repeat the algorithm from Section~\ref{sec:general}. As before, $\mathcal{K}$ denotes the cluster clique, i.e., the subgraph of $G$ corresponding to a clique $K$ in the reduced graph. Similarly, $\mathcal{B}_i$ denotes the family of cluster cliques corresponding to $B_i$ for $i\in[ 0, k-t]$. Our goal is to convert
every cluster clique $\mathcal{K} \in \mathcal{B}_1, \mathcal{B}_2, \ldots, \mathcal{B}_{k-t-1}$ to a copy of $\mathcal{R}_{\varepsilon,k-t}^{\alpha'}(L_1)$. The remaining graphs in $\mathcal{B}_{k-t}$ and $\mathcal{B}_0$ are naturally partitioned into copies of $\mathcal{R}_{\varepsilon,k-t}^{\alpha'}(L_1)$.

To explain why
this algorithm is feasible, we will reuse the calculations in Section~\ref{sec:general}. It is easy to see that if $K' \in B_{j-t}$, $K'' \in B_{k-t}$ (for $t \le j < k$), came from $K^j \in \Phi_j$ and $K^k \in \Phi_k$, respectively, and $K^j \hookrightarrow K^k$, then $K' \hookrightarrow K''$. Thus, for any $K \in B_{j-t}$, the number of well-connected $(k-t)$-cliques is at least $\Lambda(K^j) \ge \lambda_i$, if $K$ was generated from $K^j$. In the key expressions~\eqref{eq:I1} and~\eqref{eq:Ii2}, even though $s< \mu$, we use $\sigma_i = \sum_{j \ge i} j\varphi_{k-j} \ge C^{j_0}\varepsilon = \mu'$ for all $i \le t$; then $\mu'$ plays the same role as $\mu$. Specifically, the term $\mu_1$ in~\eqref{eq:mu1} is replaced by $\frac{\alpha}{1-\alpha}\mu'$. Hence, $\alpha' - \alpha = c_1\mu'$, and eventually we find an $\mathcal{R}_{\varepsilon,k-t}^{\alpha'}(L_1)$-factor of $V_B$.

Before inserting $V_0$ into $V_B$, we will move at most $2\varepsilon'|V_B|$ vertices from $V_B$ to $U_i$ if
\begin{align}\label{eq:degreevUi}
d(v, U_i) < \beta_1|U_i|,
\end{align}
where $\varepsilon' \ll \beta_1 \ll \mu'$. We may move additional vertices from $\mathcal{R}_{k-t}^{\alpha'}(L_1)$ to $V_0$ in order to maintain the ratio of cluster sizes as $(1, \ldots, 1, \alpha')$. If any cluster loses more than half of its vertices due to~\eqref{eq:degreevUi}, all vertices in that clique are removed to $V_0$. It is easy to verify that the resulting exceptional set $V_0$ satisfies $|V_0| < 2\varepsilon'n/\alpha'$.

To complete the proof of the extremal cases, we proceed by cases based on the sizes of $U_i$.

\noindent\textbf{(a) The ideal case:} $|U_i| = \left(\frac{1}{k-1} - \gamma\right)n$ for all $i\in[ 1, t]$.

Since the density between all pairs of $U_i$ is almost $1$, we first apply the Blow-up Lemma (Lemma~\ref{lemma:blowup}) to $U_1 \cup \ldots \cup U_t$ to obtain a $K_t(\omega)$-factor. Note that by~\eqref{eq:degreec}, the neighborhood of any copy of $K_t(\omega)$ in $V_B$ is almost the entire $V_B$. Let $V_{bad}=\{v\in V_B:d_G(v,A)\leq (1-\beta_2)|A|\}$,
where $\varepsilon' \ll \beta_2 \ll \beta_1$. It is easy to see that $|V_{bad}| < \frac{\varepsilon'}{\beta_2} n < \beta_2 n$.

We assume that after inserting $V_0$ into $V_B$, we obtain an $H_0$-tiling such that each copy of $H_0$ contains at most one vertex of $V_{bad}$. This assumption leads to an easy matching
between $H_0$ copies from $V_B$ and $K_t(\omega)$ copies from $V_A$. We can first use the greedy algorithm to
match those $H_0$'s that contain vertices from $V_{bad}$, then apply Hall's Theorem to handle the remaining vertices. The number of uncovered vertices in the $H$-tiling of $G$ will be proportional to the number of leftover vertices from the $H_0$-tiling in $V_B$.

To ensure each copy of $H_{k-t}$ contains at most one $V_{bad}$ vertex, we move all $V_{bad}$ vertices (and some additional vertices to maintain cluster size ratios in those cliques) to $V_0$. Since $\beta_2 \ll \mu$, we can still insert vertices in the new $V_0$ back into $V_B$.

\noindent\textbf{(b) The defective cases:} Suppose first that for some $i_1\in[1,t]$,
$|U_{i_1}|<\left(\frac{1}{k-1}-\gamma\right)n$. In this case, instead of using only copies of $K_t(\omega)$ in $V_A$, we use copies of the $t$-partite graph obtained by placing the $\sigma$-vertex class in $U_{i_1}$ and the $\omega$-vertex classes in the remaining $U_i$'s. To complete each copy of $H$, we match such a copy in $V_A$ with a copy of $K_{k-t}(\omega)$ in $V_B$. This is feasible because $\alpha'-\alpha=c_1\mu'\gg \varepsilon'$, so the tiling of $V_B$ contains enough flexibility to supply the required number of $K_{k-t}(\omega)$-copies.

If for some $i_1\in[1,t]$, $|U_{i_1}|>\left(\frac{1}{k-1}-\gamma\right)n$. As in Subcase~(I), we find $|U_{i_1}|-\left(\frac{1}{k-1}-\gamma\right)n$ vertex-disjoint $\omega$-stars inside $G[U_{i_1}]$. We then move their centers to parts whose sizes are below the ideal value, either to another $U_i$ or to $V_B$, and remove the corresponding copies of $H$. This balances the sizes of the parts up to a bounded error. After this balancing step, the argument reduces to the ideal case above.

This completes the proof of the extremal case $t<k-1$, and hence finishes the proof of Theorem~\ref{thm:maintheorem} when $\sigma<\omega$.

\section{The case when $\sigma=\omega$}\label{sec:balancedcase}

In this section, we prove Theorem~\ref{thm:main} in the case $\sigma=\omega$. In this case $H=K_k(\omega)$, $h=|V(H)|=k\omega$, and $\alpha=\sigma/\omega=1$. The Ore-type condition~\eqref{eq:degree} is equivalent to the following: for all nonadjacent vertices $x,y\in V(G)$,
\begin{align}\label{eq:degreebalanced}
d_G(x)+d_G(y)\geq 2\left(1-\frac{1}{k}\right)n.
\end{align}

The proof follows the same general strategy as in Section~\ref{sec:general}, but is simpler because the target graph is balanced. Recall that
\begin{align*}
0<\varepsilon \ll d \ll \mu \ll 1.
\end{align*}
By Lemma \ref{lemma:Rdegree}, for all nonadjacent clusters $V_i,V_j\in V(R)$,
	\[
	d_R(V_i)+d_R(V_j) \geq 2\left(1-\frac{1}{k} - d - 2\varepsilon\right)\ell.
	\]

In the balanced case, it is enough to find a $K_k$-tiling in the reduced graph $R$. The following Ore-type lemma allows us to handle the error term $2(d+2\varepsilon)\ell$.

\begin{lemma}\label{lemma:Kktiling}
	Let $G$ be a graph on $n$ vertices, and let $k \geq 2$ be an integer and $s$ be a natural number. If for all nonadjacent vertices $x,y\in V(G)$,
\[
d_G(x)+d_G(y)\geq 2 \left(1-\frac{1}{k}\right)n -2s,
\]
then $G$ has a $K_k$-tiling covering all but at most $k(k-1)s + (k-1)^2$ vertices.
\end{lemma}

\begin{proof}
Let $r$ be an integer with $0\leq r\leq k-1$ such that $\widetilde n:=n+ks+r$ is divisible by $k$. Add $ks+r$ new vertices to $G$, and join every new vertex to every other vertex. Denote the resulting graph by $\widetilde G$. For any two nonadjacent vertices $x,y\in V(\widetilde G)$, we must have $x,y\in V(G)$. Hence
\[d_{\widetilde G}(x)+d_{\widetilde G}(y)=d_G(x)+d_G(y)+2(ks+r)\geq 2\left(1-\frac{1}{k}\right)n-2s+2(ks+r)\geq 2\left(1-\frac{1}{k}\right)\widetilde n.\]
Therefore, by Theorem~\ref{thm:K2008Kr-factor}, the graph $\widetilde G$ has a $K_k$-factor.

Removing all copies of $K_k$ that contain at least one new vertex leaves a $K_k$-tiling of $G$. The number of new vertices is at most $ks+k-1$, and each copy of $K_k$ containing a new vertex contains at most $k-1$ original vertices of $G$. Therefore, the number of uncovered vertices of $G$ is at most $(k-1)(ks+k-1)=k(k-1)s+(k-1)^2$.
\end{proof}

\subsection{Decomposition lemma}

To prepare for the application of the Blow-up Lemma, we prove a balanced analogue of Lemma~\ref{lem:DS}.

\begin{lemma}\label{lemma:BDS}
For sufficiently large $\ell$, after moving at most $2k^2dn$ vertices of $G$ into the exceptional set $V_0$, the graph $G''=G\setminus V_0$ can be decomposed into vertex-disjoint balanced $\varepsilon$-regular $k$-cliques
$\mathcal{R}_{\varepsilon,k}^{1}(L)$.
\end{lemma}

\begin{proof} By Lemma~\ref{lemma:Rdegree}, for all nonadjacent $X,Y\in V(R)$, $d_R(X)+d_R(Y)\geq 2\left(1-\frac{1}{k}-d-2\varepsilon\right)\ell$. Applying Lemma~\ref{lemma:Kktiling} to $R$ with $s=(d+2\varepsilon)\ell$, we obtain a $K_k$-tiling of $R$ covering all but at most $k(k-1)(d+2\varepsilon)\ell+(k-1)^2$ vertices of $R$.

We move all clusters corresponding to these uncovered vertices of $R$ into the exceptional set $V_0$. Since $\varepsilon\leq d\ll 1$ and $\ell$ is sufficiently large, the number of vertices moved is at most
\[
\left(k(k-1)(d+2\varepsilon)\ell+(k-1)^2\right)L \leq 2k(k-1)dn \leq 2k^2dn.
\]
Thus, after updating $V_0$, we still have $|V_0|\leq 2k^2dn$.
Each $K_k$ in the $K_k$-tiling of $R$ corresponds to $k$ clusters $V_{i_1},V_{i_2},\ldots,V_{i_k}$ in $G''$, such that every pair of clusters forms an $\varepsilon$-regular pair of density greater than $d$ by the definition of the reduced graph. Since all clusters have the same size $L$, these clusters form a balanced $\varepsilon$-regular $k$-clique $\mathcal{R}_{\varepsilon,k}^{1}(L)$.
The collection of all such regular $k$-cliques gives the desired decomposition of $G''$.
\end{proof}

\subsection{Handling exceptional vertices and bounding the leftover vertices}

The following lemma is the balanced-case analogue of Lemma~\ref{lem:generalcase}.

\begin{lemma}\label{lemma:balanced-exceptional}
Assume that $\varepsilon\leq \theta\ll 1$. Let $G$ be an $n$-vertex graph satisfying~\eqref{eq:degreebalanced}, and let $V_0\subseteq V(G)$ be an exceptional set with $|V_0|\leq \theta n$. If $G''=G\setminus V_0$ has an $\mathcal{R}_{\varepsilon,k}^{1}(L_1)$-factor for sufficiently large $L_1$. Then $G$ contains an $H$-tiling covering all but at most $3k^2\omega+h$ vertices.
\end{lemma}

We use notation consistent with Subsection~\ref{subsec:exvtx}.
\begin{itemize}
	\item Let $\mathbf{R}$ be the $\mathcal{R}_{\varepsilon,k}^1(L_1)$-factor of $G''$, with each clique $\mathcal{R}\in \mathbf{R}$ having color-classes $U_1,U_2,\ldots,U_k$ and $|U_i|=L_1$ for all $i\in[1,k]$.
	\item For a vertex $v\in V_0$ and a cluster $C\in V(R)$, we write $v\sim C$ if $d_G(v,C)\geq d|C|$.
\end{itemize}

\begin{proof}[Proof of Lemma~\ref{lemma:balanced-exceptional}] We partition $V_0$ into two subsets. Let $V_0^1=\left\{x\in V_0:d_G(x)<\left(1-\frac{1}{k}\right)n\right\}$ and let
$V_0^2=V_0\setminus V_0^1$. By the Ore-type condition~\eqref{eq:degreebalanced}, the set $V_0^1$ induces a clique in $G$. Hence $G[V_0^1]$ can be tiled by copies of $H=K_k(\omega)$, leaving at most $h-1$ vertices uncovered. It remains to cover the vertices of $V_0^2$.

\noindent\textbf{Phase I: Inserting exceptional vertices.}

We process the vertices of $V_0^2$ in an arbitrary order. For each $v\in V_0^2$, we choose an available clique $\mathcal{R}=\{U_1,\ldots,U_k\}\in\mathbf{R}$ such that
\begin{equation}\label{eq:adj-condition}
v\sim U_i
\quad\text{for all but at most one cluster }U_i\in\{U_1,\ldots,U_k\}.
\end{equation}

Lemma~\ref{lem:komlos1995} guarantees that we can construct a copy of $H$ containing $v$ as follows.
\begin{itemize}
	\item If $v\sim U_i$ for all $i\in[1,k]$, choose any cluster, say $U_1$, to contain the color-class of $H$ containing $v$. We take $\omega-1$ additional vertices from $U_1$, and $\omega$ vertices from each of the other $k-1$ clusters.
	\item If $v\nsim U_j$ for exactly one $j\in[1,k]$, then we place $v$ in the color-class corresponding to $U_j$. We take $\omega-1$ additional vertices from $U_j$, and $\omega$ vertices from each of the other $k-1$ clusters.
\end{itemize}

To prevent a cluster from losing too many vertices, we stop considering an element of $\mathbf{R}$ if it has been selected $\theta_1 L_1$ times, where $\theta_1=\sqrt{\theta}$.
The following proposition guarantees that sufficiently many valid cliques are always available.

\begin{prop}\label{prop:valid-clique-bound}
For every vertex $v\in V_0^2$, regardless of the order in which the vertices of $V_0^2$ are processed, at least a $1/3$-proportion of the elements of $\mathbf{R}$ satisfy~\eqref{eq:adj-condition}.
\end{prop}

\begin{proof}
Let $v\in V_0^2$ be the $(t+1)$-st vertex to be processed, where $t$ denotes the number of vertices of $V_0^2$ that have already been processed. For these first $t$ vertices, we have removed $t$ pairwise disjoint copies of $H$, each containing exactly one processed vertex of $V_0^2$ and $h-1$ vertices from elements of $\mathbf{R}$.

Let $G^*$ be the subgraph induced by the vertices remaining after these $t$ copies have been removed. Since $v\in V_0^2$, we have $d_G(v)\geq \left(1-\frac{1}{k}\right)n$.
Moreover, at most $|V_0|\leq \theta n$ vertices lie in the exceptional set, and at most $t(h-1)$ vertices have been removed from elements of $\mathbf{R}$. Hence
\begin{equation}\label{eq:balanced-lower}
d_G(v,V(G^*))\geq\left(1-\frac{1}{k}\right)n-\theta n-t(h-1).
\end{equation}

Let $m$ be the fraction of cliques in $\mathbf{R}$ satisfying~\eqref{eq:adj-condition}. If a clique does not satisfy~\eqref{eq:adj-condition}, then $v$ fails to be $d$-adjacent to at least two of its clusters. Hence such a clique contributes at most $(k-2+2d)L_1$ neighbors of $v$. A clique satisfying~\eqref{eq:adj-condition} contributes at most $kL_1$ neighbors. Therefore,
\begin{equation}\label{eq:balanced-upper}
d_G(v,V(G^*))\leq\left((1-m)\frac{k-2+2d}{k}+m\right)(n-\theta n-t(h-1)).
\end{equation}

Combining~\eqref{eq:balanced-lower} and~\eqref{eq:balanced-upper}, we obtain $1-\frac{1}{k}-\theta h\leq(1-m)\frac{k-2+2d}{k}+m$.
Since $d,\theta\ll 1$, so $m\geq\frac{1-2d-k\theta h}{2-2d}\ge 1/3$.
\end{proof}

The proportion of unavailable cliques is at most
\[
\frac{|V_0|/(\theta_1L_1)}{|\mathbf{R}|}\leq\frac{\theta n/(\sqrt{\theta}L_1)}{n/(kL_1)}=k\sqrt{\theta}.
\]
Since $k\sqrt{\theta}<1/3$, Proposition~\ref{prop:valid-clique-bound} ensures that an available clique satisfying~\eqref{eq:adj-condition} exists for every vertex of $V_0^2$. Thus all vertices of $V_0^2$ can be covered.
After processing all vertices of $V_0^2$, at most $h-1$ vertices from $V_0^1$ remain uncovered.

We now make all regular pairs super-regular. For each clique $\mathcal{R}=\{U_1,\ldots,U_k\}\in\mathbf{R}$ and each vertex $u\in U_i$, if there exists $j\neq i$ such that $d_G(u,U_j)<(d-\varepsilon)|U_j|$, then we move $u$ into the exceptional set. By $\varepsilon$-regularity, each cluster loses at most $(k-1)\varepsilon |U_i|$ vertices in this way. Hence the number of newly exceptional vertices is at most $k(k-1)\varepsilon n$. Since $\varepsilon\ll\theta$, these new exceptional vertices can be processed by the same argument as above. After doing this, every pair inside each remaining clique is $(\varepsilon,d/2)$-super-regular.

Let $\mathbf{R}'$ denote the resulting family of super-regular balanced $k$-cliques.

\noindent\textbf{Phase II: Bounding the leftover vertices.}

We first bound the number of uncovered vertices in a single super-regular balanced clique.

\begin{prop}\label{prop:single-clique-leftover}
	Each element of $\mathbf{R}'$ has an $H$-tiling that covers all but at most $k(\omega-1)$ vertices.
\end{prop}
\begin{proof}
Let $\mathcal{R}'=\{U_1,\ldots,U_k\}\in\mathbf{R}'$. Take a maximum $H$-tiling of $\mathcal{R}'$, and let $a_i$ be the number of uncovered vertices in $U_i$. We claim that $a_i<\omega$ for every $i\in[1,k]$. Otherwise, if $a_i\geq \omega$ for every $i$, then the uncovered vertices contain $\omega$ vertices from each cluster. Since the pairs inside $\mathcal{R}'$ are super-regular, the Blow-up Lemma gives one more copy of $H=K_k(\omega)$, contradicting the maximality of the tiling. Hence $a_i\leq \omega-1$ for all $i\in[1,k]$. Therefore, $\sum_{i=1}^{k}a_i\leq k(\omega-1)$.
This proves the proposition.
\end{proof}

Applying Proposition~\ref{prop:single-clique-leftover} independently to every element of $\mathbf{R}'$ gives an $H$-tiling with at most $k(\omega-1)|\mathbf{R}'|$ uncovered vertices. We now use the same directed-graph transfer argument as in Subsection~\ref{subsec:exvtx} to reduce this to a constant independent of $\varepsilon$.

Define a directed graph $\mathcal{D}$ on the clusters of $G''$. Let $U$ be a cluster, and let $\mathcal{R}=\{U_1,\ldots,U_k\}\in\mathbf{R}'$. There is an edge $(U,U_1)\in E(\mathcal{D})$ if $(U,U_i)\in E(R)$
for all $i\neq 1$. The adjacencies in $R$ are not affected by the vertex removals in Phase I, so $\mathcal{D}$ is well-defined.

We next estimate the out-degree of high-degree clusters.

\begin{prop}\label{prop:out-degree-bound}
	For any $U \in V(\mathcal{D})$ with $d_R(U) \geq \left(1 - \frac{1}{k} - d - 2\varepsilon\right)|V(R)|$, we have $d_{\mathcal{D}}^{+}(U)\geq \frac{1}{2k}|V(\mathcal{D})|$.
\end{prop}

\begin{proof}
	For a given $U \in V(\mathcal{D})$, let $m_1$, $m_2$ denote the fraction of cliques $\mathcal{R}=\{U_1, U_2, \ldots, U_k\} \in \mathbf{R}$ for which $(U, U_i) \in E(R)$ but $(U, U_j) \notin E(R)$ for all $i \neq j$ with $j \neq k$, and $(U, U_i) \in E(R)$ for all $i\in [1,k]$, respectively. By the same degree counting as in Proposition~\ref{prop:degreeU}, the degree condition of the reduced graph $d_R(U)\geq \left(1-\frac{1}{k} -d-2\varepsilon\right)\ell$ gives
	\[
	1-\frac{1}{k} - d-2\varepsilon \leq \frac{k-2}{k} +\frac{m_1}{k} +\frac{2m_2}{k}.
	\]
Therefore, $m_1+2m_2\geq 1-k(d+2\varepsilon)\geq \frac{1}{2}$, provided $d$ and $\varepsilon$ are sufficiently small.

By the definition of $\mathcal{D}$, a clique counted by $m_1$ gives one out-neighbor of $U$, while a clique counted by $m_2$ gives $k$ out-neighbors of $U$. Since $|\mathbf{R}'|=|V(\mathcal{D})|/k$, we obtain \begin{equation*}d_{\mathcal{D}}^{+}(U)=(m_1+km_2)|\mathbf{R}'|\geq
(m_1+2m_2)|\mathbf{R}'|\geq\frac12|\mathbf{R}'|\geq\frac{1}{2k}|V(\mathcal{D})|.\qedhere
\end{equation*}
\end{proof}

Let $J_L=\left\{U\in V(\mathcal{D}):d_R(U)<\left(1-\frac{1}{k}-d-2\varepsilon\right)|V(R)|\right\}$, and let $
J_H=V(\mathcal{D})\setminus J_L$. By the Ore-type condition on $R$, the set $J_L$ induces a clique in $R$. Thus any leftover vertices transferred into clusters of $J_L$ can be handled inside this clique structure, leaving at most $k$ additional vertices uncovered.

For $J_H$, Proposition~\ref{prop:out-degree-bound} and Lemma~\ref{lem:sinkset} imply that the high-degree part of the sink set has size at most $\frac{|V(\mathcal{D})|}{\frac{1}{2k}|V(\mathcal{D})|+1}\leq2k$.
We transfer leftover vertices as in Subsection~\ref{subsec:exvtx}. For each cluster $C$, let $\operatorname{extra}(C)$ be the number of vertices left uncovered by the tiling from Proposition~\ref{prop:single-clique-leftover}. If $C$ is not in the sink set, choose a directed path from $C$ to the sink set and move the extra vertices along this path. Each directed edge allows us to remove copies of $H$ using one vertex from the preceding cluster and $\omega$ vertices from each of the other $k-1$ clusters of the next clique. Since the total number of transferred vertices is bounded by a constant and $L_1$ is sufficiently large, this does not affect super-regularity.

After all extra vertices have been transferred, all leftover vertices are concentrated in at most $2k$ high-degree sink clusters, together with at most $k$ additional vertices arising from the low-degree clique structure. Hence, by Proposition~\ref{prop:single-clique-leftover}, the total number of uncovered vertices is at most $k(\omega-1)\cdot 2k+k+h\leq3k^2\omega+h$.

This proves Lemma~\ref{lemma:balanced-exceptional}.
\end{proof}

Combining Lemma~\ref{lemma:BDS} and Lemma~\ref{lemma:balanced-exceptional}, we obtain an $H$-tiling of $G$ covering all but at most $3k^2\omega+h$ vertices. Since $H$ is fixed, this is a constant depending only on $H$. This completes the proof of Theorem~\ref{thm:maintheorem} in the case $\sigma=\omega$.

\section{Concluding remarks}\label{sec:Remarks}

In this paper, we strengthen Theorem~\ref{thm:SZ2003H-matching} of Shokoufandeh and Zhao by replacing the minimum-degree condition with an Ore-type condition. This also resolves a conjecture proposed by K\"uhn, Osthus, and Treglown~\cite{kuhn2009ore}.
In the same paper, K\"uhn, Osthus, and Treglown raised several problems concerning the existence of perfect tilings in graphs with respect to the critical chromatic number. Relatedly, K\"uhn and Osthus~\cite{kuhn2006} obtained a minimum-degree condition for perfect tilings and provided a polynomial-time algorithm. It would be interesting to establish Ore-type conditions for perfect tilings involving the critical chromatic number, from both theoretical and algorithmic perspectives.

\bibliographystyle{abbrv}
\bibliography{AlonYuster}

\end{document}